\documentclass[final]{siamart1116}
\usepackage{amssymb}
\usepackage{amsmath}
\usepackage{amsfonts} 
\usepackage{amsbsy}
\usepackage{amscd}
\usepackage{fullpage}
\usepackage{times}
\usepackage{psfrag}
\usepackage{graphics}
\usepackage{color}
\usepackage{bbm}
\usepackage{colortbl}
\usepackage{graphicx}
\usepackage{caption}
\usepackage{bbm}
\usepackage{subcaption}
\usepackage{array,supertabular}
\usepackage{tabularx}
\usepackage{booktabs}
\usepackage{multirow}
\usepackage{ulem}
\usepackage{units}
\usepackage{mathabx}
\usepackage{accents}
\usepackage{algorithmicx}
\usepackage{relsize}
\usepackage{algorithm}
\usepackage{xfrac}
\usepackage{algpseudocode}
\usepackage{rotating}
\usepackage[capitalize]{cleveref}
 \usepackage[notref,notcite]{showkeys}
\usepackage{mathrsfs} 
\usepackage[framemethod=tikz]{mdframed} 
\usepackage{circuitikz}
\usepackage{appendix}

\newcommand{\x}{\mathbf{x}}

\newcommand{\y}{\mathbf{y}}
\newcommand{\bb}{\mathbf{b}}
\newcommand{\cc}{\mathbf{c}}

\newcommand{\0}{\mathbf{0}}

\newcommand{\ds}{\displaystyle}

\crefname{secinapp}{Section}{Sections}
\Crefname{secinapp}{Section}{Sections}

\newlength{\dhatheight}

\newsiamremark{remark}{Remark}
\newsiamthm{conjecture}{Conjecture}
\newcommand{\newsiamthmnn}[2]{
  \theoremstyle{nonumberplain}
  \theoremheaderfont{\normalfont\sc}
  \theorembodyfont{\normalfont\itshape}
  \theoremseparator{.}
  \theoremsymbol{}
  \newtheorem{#1}[theorem]{#2}
}

\newsiamthmnn{main1}{Theorem \ref{thm:main}}
\newsiamthmnn{main2}{Theorem \ref{cor:onefifth}}

\title{Resistance distance in straight linear 2-trees}
\author{Wayne Barrett \thanks{Department of Mathematics,
  Brigham Young University,
  Provo, Utah 84602, USA\funding{Supported by the Defense Threat Reduction Agency -- Grant Number HDTRA1-15-1-0049}}
  \and
Emily. J. Evans \thanks{Department of Mathematics,
  Brigham Young University,
  Provo, Utah 84602, USA\funding{Supported by the Defense Threat Reduction Agency -- Grant Number HDTRA1-15-1-0049}}
  \and
Amanda E. Francis \thanks{Department of Mathematics, Engineering, and Computer Science, 
  Carroll College,
  Helena, Montana 59625, USA\funding{Supported by the Defense Threat Reduction Agency -- Grant Number HDTRA1-15-1-0049}}}
\begin{document}
\maketitle
\begin{abstract}
%
We consider the graph $G_n$ with vertex set $V(G_n) = \{ 1, 2, \ldots, n\}$ and $\{i,j\} \in E(G_n)$ if and only if $0<|i-j| \leq 2$.  
We call $G_n$ the straight linear 2-tree on $n$ vertices.   Using $\Delta$--Y transformations and identities for the Fibonacci and Lucas numbers we obtain explicit formulae for the resistance distance $r_{G_n}(i,j)$ between any two vertices $i$ and $j$ of $G_n$. To our knowledge $\{G_n\}_{n=3}^\infty$ is the first nontrivial family with diameter going to $\infty$ for which all resistance distances have been explicitly calculated. Our result also gives formulae for the number of spanning trees and 2-forests in a straight linear 2-tree.  We show that the maximal resistance distance in $G_n$ occurs between vertices 1 and $n$ and the minimal resistance distance occurs between vertices $n/2$ and $n/2+1$ for $n$ even (with a similar result for $n$ odd).  It follows that $r_n(1,n) \to \infty$ as $n \to \infty$.  Moreover, our explicit formula makes it possible to order the non-edges of $G_n$ exactly according to resistance distance, and this ordering agrees with the intuitive notion of distance on a graph. Consequently, $G_n$ is a geometric graph with entirely different properties than the random geometric graphs investigated in [6].  These results for straight linear 2-trees along with an example of a bent linear 2-tree and empirical results for additional graph classes convincingly demonstrate that resistance distance should not be discounted as a viable method for link prediction in geometric graphs.  
\end{abstract}

\begin{keyword}
effective resistance, resistance distance, 2--tree, spanning tree
\end{keyword}

\section{Introduction}
The resistance distance (also known as the effective resistance) of a graph is one measure of quantifying structural properties of a given graph. 
Roughly speaking the resistance distance between two nodes on a graph is determined by considering the graph as an electrical circuit with edges being represented as resistors.  Using the standard laws of electrical conductance, the resistance distance of the graph can be determined.  For undirected graphs, a well known result links the resistance distance between two nodes to the Moore-Penrose inverse of the graph Laplacian. It is also well known that the resistance distance of a graph is directly related to the commute time in a graph.  For a few graphs, such as the wheel and the fan, resistance distance between any two vertices has been calculated explicitly as a function of the number of vertices in the graph~\cite{BapatWheels, YangKlein}.

The resistance distance has been used by several authors to perform operations on graphs and to quantify graph behavior.   Spielman and Srivastava~\cite{SpielSparse} have used resistance distance between nodes of graphs to develop an algorithm to rapidly sparsify a given graph while maintaining spectral properties.  Resistance distance is also seen in the field of distributed control and estimation.  In particular, one problem in this field is the estimation of several variables in the presence of noisy data. This is of particular interest in the design and operation of sensor arrays~\cite{Barooah06grapheffective}.

Recently, concerns were raised that resistance distance fails a number of desirable properties of a distance function for certain random geometric graphs \cite{lostinspace}. For these graphs they obtain the asymptotic result that 
\begin{equation}\label{eqn_LIS}
r_G(i,j) \approx \frac{1}{\deg(i)}  + \frac{1}{\deg(j)}.
\end{equation}
Since the value of $r_G(i,j)$ here depends only on the degrees of vertices $i$ and $j$, they conclude that $r_G(i,j)$ is completely meaningless as a distance function on these large geometric graphs. 

Of course, the preceding result does not hold for some classes of graphs.  For trees $r_G(i,j) = d_G(i,j)$, so $r_G(i,j)$ is still a distance function.  Although resistance distance has been calculated for a number of special graphs \cite{BapatWheels, YangKlein}, there seems to be a paucity of results for infinite classes of graphs. (Some exceptions are paths, wheels, and fans.) We investigate another infinite class of 2-trees in this paper for which $r_G(i,j)$ retains all desirable properties of a distance function. 

\begin{definition}[straight linear 2-tree]\label{def:lin2treest}
A straight linear 2-tree is a graph $G_n$ with $n$ vertices with adjacency matrix that is symmetric, banded, with the first and second subdiagonals equal to one, and first and second superdiagonals equal to one, and all other entries equal to zero. See Figure~\ref{fig:2tree}. \end{definition} 

\begin{figure}[h!]
\begin{center}

\begin{tikzpicture}[line cap=round,line join=round,>=triangle 45,x=1.0cm,y=1.0cm,scale = 1.2]
\draw [line width=1.pt] (-3.,0.)-- (-2.,0.);
\draw [line width=1.pt] (-2.,0.)-- (-1.,0.);
\draw [line width=1.pt,dotted] (-1.,0.)-- (0.,0.);
\draw [line width=1.pt] (0.,0.)-- (1.,0.);
\draw [line width=1.pt] (1.,0.)-- (2.,0.);
\draw [line width=1.pt] (2.,0.)-- (1.5,0.866025403784435);
\draw [line width=1.pt] (1.5,0.866025403784435)-- (1.,0.);
\draw [line width=1.pt] (1.,0.)-- (0.5,0.8660254037844366);
\draw [line width=1.pt] (0.5,0.8660254037844366)-- (0.,0.);
\draw [line width=1.pt] (-0.5,0.8660254037844378)-- (-1.,0.);
\draw [line width=1.pt] (-1.,0.)-- (-1.5,0.8660254037844385);
\draw [line width=1.pt] (-1.5,0.8660254037844385)-- (-2.,0.);
\draw [line width=1.pt] (-2.,0.)-- (-2.5,0.8660254037844388);
\draw [line width=1.pt] (-2.5,0.8660254037844388)-- (-3.,0.);
\draw [line width=1.pt] (-2.5,0.8660254037844388)-- (-1.5,0.8660254037844385);
\draw [line width=1.pt] (-1.5,0.8660254037844385)-- (-0.5,0.8660254037844378);
\draw [line width=1.pt,dotted] (-0.5,0.8660254037844378)-- (0.5,0.8660254037844366);
\draw [line width=1.pt] (0.5,0.8660254037844366)-- (1.5,0.866025403784435);
\begin{scriptsize}
\draw [fill=black] (-3.,0.) circle (1.5pt);
\draw[color=black] (-3.02279181666165,-0.22431183338253265) node {$1$};
\draw [fill=black] (-2.,0.) circle (1.5pt);
\draw[color=black] (-2.0001954862580344,-0.22395896857501957) node {$3$};
\draw [fill=black] (-2.5,0.8660254037844388) circle (1.5pt);
\draw[color=black] (-2.5018465162673555,1.100526386601008717) node {$2$};
\draw [fill=black] (-1.5,0.8660254037844385) circle (1.5pt);
\draw[color=black] (-1.5081915914412003,1.100526386601008717) node {$4$};
\draw [fill=black] (-1.,0.) circle (1.5pt);
\draw[color=black] (-1.0065405614318794,-0.22290037415248035) node {$5$};
\draw [fill=black] (-0.5,0.8660254037844378) circle (1.5pt);
\draw[color=black] (-0.4952423962300715,1.100526386601008717) node {$6$};
\draw [fill=black] (0.,0.) circle (1.5pt);
\draw[color=black] (-0.03217990699069834,-0.22431183338253265) node {$n-4$};
\draw [fill=black] (0.5,0.8660254037844366) circle (1.5pt);
\draw[color=black] (0.4887653934035965,1.103344389715898) node {$n-3$};
\draw [fill=black] (1.,0.) circle (1.5pt);
\draw[color=black] (0.9904164234129174,-0.22431183338253265) node {$n-2$};
\draw [fill=black] (1.5,0.866025403784435) circle (1.5pt);
\draw[color=black] (1.5692445349621338,1.1042991524908385) node {$n-1$};
\draw [fill=black] (2.,0.) circle (1.5pt);
\draw[color=black] (1.993718483431559,-0.22431183338253265) node {$n$};
\end{scriptsize}
\end{tikzpicture}
%
%
\end{center}
\caption{A straight linear 2-tree}
\label{fig:2tree}
\end{figure}
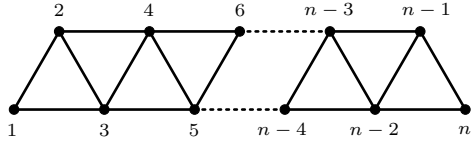

\begin{remark} We observe that $G_n$ is a geometric graph. This can easily be seen by placing the vertices so that all of the triangles in Figure~\ref{fig:2tree} are equilateral.
\end{remark}

Our main result is: 

\begin{main1} 
Let $G_n$ be the straight linear 2-tree on $n$ vertices labeled as in Figure \ref{fig:2tree}, and let $m = n-2$ be the number of triangles in $G_n$.  Then for any two vertices $j$ and $j+k$ of $G_n$, 
\[
r_m(j,j+k) = \frac{\sum_{i=1}^{k} \left( F_iF_{i+2j-2} - F_{i-1}F_{i+2j-3}\right)F_{2m-2i-2j+5}}{F_{2m+2}},
\] or in closed form 
	 \[r_m(j,j+k)=\frac{F_{m+1}^2+F_k^2F_{m-2j-k+3}^2+\frac{F_{m+1}}{5}\left[F_{m-k}(kL_k-F_k)+F_{m-k+1}\left((k-5)F_{k+1}+(2k+2)F_{k}\right)\right]}{F_{2m+2}}\]
\noindent where $F_p$ is the $p$th Fibonacci number and $L_k$ is the $k$th Lucas number. 
\end{main1}

An important consequence of the first formula is Corollary~\ref{cor:monotonicity}, which verifies that resistance distance behaves as a distance function.  There is one additional property that $r_m(j,j+k)$ should have to be a genuine distance function, namely $r_m(j,j+k)$ should increase without bound as $k \to \infty$. 

It is easy to see from the second formula that  
\begin{main2}
Let $G_n$ and $G_{n+1}$ be the straight linear 2-trees on $n$ and $n+1$ vertices respectively.  Then 
\[
\lim_{n \to \infty} [r_{m+1}(1, n+1) - r_m(1,n)] = \frac{1}{5}.
\]
In particular, $\displaystyle{\lim_{n \to \infty} r_m(1,n) = \infty}$. 

\end{main2}

Now suppose $H_n$ is any connected subgraph of $G_n$ containing the two degree two vertices 1 and $n$. Recall Rayleigh's monotonicity law \cite[Lemma D]{KleinRandic} which states that if $P$ is an electrical circuit with a resistance on each edge, and a new circuit $K$ is created from $P$ by lowering the resistance on an existing edge, or inserting a new edge, then $r_K(i,j) \leq r_H(i,j)$ for
 all vertices $i$ and $j$ of $H$.  Applying this law we have the following result
\[
r_{H_n}(1,n) \geq r_{G_n}(1,n).
\]
So, as a consequence 
\begin{corollary}
Let $G_n$ be the linear 2-tree in Figure \ref{fig:2tree} and let $H_n$ be any connected subgraph of $G_n$ containing vertices 1 and $n$.  Then 
\[
\lim_{n \to \infty} r_{H_n}(1,n) = \infty. 
\]

\end{corollary}

In particular, neither the geometric graph $G_n$, nor any of its connected subgraphs containing vertices 1 and $n$, exhibits the behavior in Equation \ref{eqn_LIS}. 

We conjecture that similar asymptotic behavior holds for the straight linear $k$-trees and its subgraphs, where the straight linear $k$-tree is the graph with vertex set $\{1, 2, \ldots, n\}$ and $\{i,j\} \in E(G)$ if and only if $0<|i-j| \leq k$. This leads us to believe there are large classes of geometric graphs for which the results in \cite{lostinspace} do not hold. Further investigation is needed, but we believe that resistance distance still deserves consideration as an effective tool for link prediction and other applications. 

A major question is: Which real-life networks behave like subgraphs of straight linear 2-trees (or $k$-trees), and which behave like the random geometric graphs in \cite{lostinspace}? 

%
%
%
%
In Section~\ref{sec:note} we provide needed background, explain our main tool, the $\Delta$--Y transform, and cite important Fibonacci and Lucas number identities that will be used in the paper. In Section 3 we illustrate our method for determining resistance distance between any two nodes in a straight linear 2-tree by considering the simpler problem of determining the resistance distance between the extreme vertices.  Next, in Section 4 we give the proof of our main result.  In Section 5 we apply our main result to give an explicit formula for the number of spanning trees and 2-forests in a straight linear 2-tree. As we will show, the denominator in the formulae in Theorem~\ref{thm:main} is the number of spanning trees of $G_n$ matching the results for fans in \cite{BapatWheels}. Next, in Section 6, we use our main result to obtain a monotonicity result and apply this result to the link prediction problem.  Finally, in Section 7 we give some conclusions and some conjectures.

\section{Notation and Necessary Identities}\label{sec:note}
An undirected graph $G$ consists of a finite set $V$ called the vertices, a subset $E$ of two-element subsets of $V$ called the edges, and a set $w$ of positive weights associated with each edge in $G$.  Our focus is on connected simple graphs, although we occasionally allow a single multiple edge when calculating resistance using the parallel rule.  For convenience we usually take the vertex set $V$ to be $\{1, 2, \ldots, n\}$. 

The adjacency matrix of $G$, $A(G)$, is the $n\times n$ nonnegative symmetric matrix defined by 
\[
a_{ij} = \left\{ \begin{array}{cl} w(i,j) & \text{if } \{i,j\} \text{ is an edge of }G\\
0 & \text{otherwise} \end{array}\right. 
\]
and its Laplacian matrix, $L(G)$ is the $n\times n$ real symmetric matrix defined by 
\[
\ell_{ij} = \left\{ \begin{array}{cl} \deg(i) & \text{if } i = j \\
-w(i,j) & \text{if } i \neq j \text{ and } \{i,j\} \text{ is an edge of } G\\
0 & \text{otherwise,}\end{array}\right.
\]
where $\deg(i)$ is the sum of the edge weights of the edges incident to $i$.  The Laplacian matrix of a connected graph is a singular positive semidefinite matrix whose null space is spanned by the all ones vector. 

Now assume that the graph $G$ represents an electrical circuit with resistances on each edge.  The resistance on a weighted edge is the reciprocal of its edge weight. Given any two nodes $i$ and $j$ assume that one unit of current flows into node $i$ and one unit of current flows out of node $j$.  The potential difference $v_i - v_j$ between nodes $i$ and $j$ needed to maintain this current is the \emph{effective resistance} or \emph{resistance distance} between $i$ and $j$ since by Ohm's law, the resistance, $r_G(i,j)$ is 
\[
r_G(i,j) = \frac{v_i-v_j}{1}=v_i-v_j.
\]
The following method for calculating $r(i,j)$ works for every circuit.  Let $v_i$ be the potential at node $i$ and let $\mathbf{v} = [v_1, \ldots, v_n]^T$ be the $n$-vector of voltages.  The source vector of currents is $\mathbf{s} = \mathbf{e}_i -\mathbf{e}_j$.  Then by Kirchoff's law and Ohm's law, 
\[
L(G)\mathbf{v} = \mathbf{s}.
\]

It follows that for any generalized inverse $X$ of $L(G)$, $\mathbf{v} = X \mathbf{s}$. Then 
\[
r_G(i,j) ={v}_i - {v}_j = (\mathbf{e}_i - \mathbf{e}_j)^T X (\mathbf{e}_i - \mathbf{e}_j).
\]
Usually $X$ is taken to be $L^\dagger$, the Moore-Penrose inverse of $L(G)$, so we have 
\begin{equation}
r_G(i,j) = (\mathbf{e}_i - \mathbf{e}_j)^T L^\dagger (\mathbf{e}_i - \mathbf{e}_j).
\end{equation}
This formula works extremely well for small and intermediate sized graphs, but is difficult to apply both theoretically and computationally for very large graphs.  However, many other techniques can be employed to calculate resistance distance, including the well-known series and parallel rules and the $\Delta$--Y transformation. 

%



\begin{definition}[Series Transformation] Let $N_1$, $N_2$, and $N_3$ be nodes in a graph where $N_2$ is adjacent to only $N_1$ and $N_3$.  Moreover, let $R_A$ equal the resistance between $N_1$ and $N_2$ and $R_B$ equal the resistance between node $N_2$ and $N_3$.  A series transformation transforms this graph by deleting $N_2$ and setting the resistance between $N_1$ and $N_3$ equal to $R_A + R_B$.\end{definition}

\begin{definition}[Parallel Transformation] Let $N_1$ and $N_2$ be nodes in a multi-edged graph where $e_1$ and $e_2$ are two edges between $N_1$ and $N_2$ with resistances $R_A$ and $R_B$, respectively.   A parallel transformation transforms the graph by deleting edges $e_1$ and $e_2$ and adding a new edge between $N_1$ and $N_2$ with edge resistance $r = \left(\frac{1}{R_A} + \frac{1}{R_B}\right)^{-1}$.
\end{definition}

A $\Delta$--Y transformation is a mathematical technique to convert resistors in a triangle ($\Delta$) formation to an equivalent system of three resistors in a ``Y'' format as illustrated in Figure~\ref{fig:dy}.  We formalize this transformation below. 
\begin{definition}[$\Delta$--Y transformation]\label{def:dy}
Let $N_1, N_2, N_3$ be nodes and $R_A$, $R_B$ and $R_C$ be given resistances as shown in Figure~\ref{fig:dy}.  The transformed circuit in the ``Y'' format as shown in Figure~\ref{fig:dy} has the following resistances:
\begin{align}
  R_1 &= \frac{R_BR_C}{R_A + R_B + R_C} \\
  R_2 &= \frac{R_AR_C}{R_A + R_B + R_C} \\
  R_3 &= \frac{R_AR_B}{R_A + R_B + R_C}
\end{align}
\end{definition}
 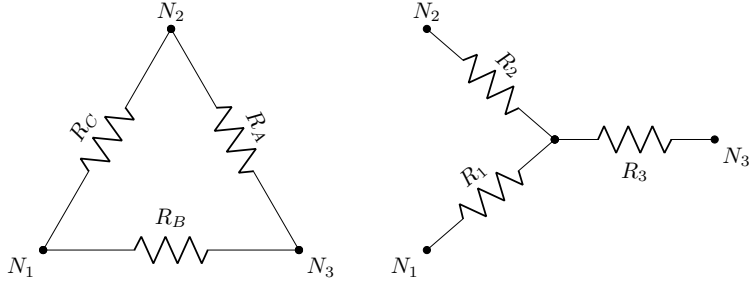
\begin{figure}
 \begin{center}
\resizebox{4in}{!}{ \begin{circuitikz}
 
 \draw (0,0) to [R, *-*,l=$R_B$] (4,0);
  \draw (0,0) to [R, *-*,l=$R_C$] (2,3.46);
  \draw (2,3.46) to[R, *-*,l=$R_A$] (4,0)
  {[anchor=north east] (0,0) node {$N_1$} } {[anchor=north west] (4,0) node {$N_3$}} {[anchor=south]  (2,3.46) node {$N_2$}};
  
  \draw (6,0) to [R, *-*,l=$R_1$] (8,1.73);
  \draw (6,3.46) to [R, *-*,l=$R_2$] (8,1.73);
 \draw (10.5,1.73) to [R, *-*,l=$R_3$] (8,1.73)
  {[anchor=north east] (6,0) node {$N_1$} } {[anchor=north west] (10.5,1.73) node {$N_3$}} {[anchor=south]  (6,3.46) node {$N_2$}};
 ; 
 \end{circuitikz}}
 \end{center}
 \caption{ $\Delta$ and $Y$ circuits with vertices labeled as in Definition~\ref{def:dy}.}
 \label{fig:dy}
 \end{figure}
\begin{proposition} Series transformations, parallel transformations, and $\Delta$--Y transformations yield equivalent circuits.
\end{proposition}
\begin{proof}
See~\cite{oldbook} for a proof of this result.
\end{proof}

In addition to the network transformations just described, we will also use the following cut-vertex theorem to calculate resistance distances.  It is evident that this result must hold from the physical meaning of resistance distance.  However, we were unable to find this exact statement in past work, so we have included it for completeness.
\begin{theorem}[Cut Vertex]
Let $G$ be a connected graph with weights $w(i,j)$ for every edge of $G$ and suppose $v$ is a cut-vertex of $G$.  Let $C$ be a component of $G - v$ and let $H$ be the induced subgraph on $V(C) \cup \{v\}$. Then for each pair of vertices $i$, $j$ of $H$, 
\begin{equation}\label{eq:star}
r_G(i,j) = r_H(i,j).
\end{equation}
\end{theorem}
\begin{proof}
Equation~\ref{eq:star} is trivially true if $i = j$, so assume $i \neq j$.  Let $V(G) = \{1, 2, \ldots, n\}$.  Without loss of generality we take $V(H) = \{ 1, 2, \ldots, k\}$. 
 
\noindent \textbf{Case 1: } Neither $i$ nor $j$ is $v$.  Without loss of generality we take $i$ and $j$ to be vertices 1 and 2, and $v$ to be vertex $k$.  We may then write the Laplacian matrices of $G$ and $H$ as 
\[
L_G = \left[ \begin{array}{ccccc}
d_1 & -a & -\bb^T & -g & 0 \\
-a & d_2 & -\cc^T & -h & 0 \\ 
-\bb & -\cc & D & -\x & 0 \\
-g & -h & -\x^T & d_k & -\y^T \\
0 & 0 & 0 & -\y & E
\end{array}\right], \ \ L_H = \left[ \begin{array}{cccc}
d_1 & -a & -\bb^T & -g  \\
-a & d_2 & -\cc^T & -h  \\ 
-\bb & -\cc & D & -\x  \\
-g & -h & -\x^T & d_k^{\;\prime}  \\
\end{array}\right], 
\]

\noindent Since $d_k = g + h + \x^T \mathbf{1} + \y^T \mathbf{1}$ and 
$d_k^{\;\prime} = g + h + \x^T \mathbf{1}$, we have $d_k = d_k^{\;\prime} + \y^T\mathbf{1}$. By Lemma 1 in \cite{BapatWheels},
\begin{equation}\label{eq:detfrac}
r_G(1,2) =\frac{\det \left[ \begin{array}{ccc}
 D & -\x & 0 \\
-\x^T & d_k & -\y^T \\
 0 & -\y & E\end{array}\right]}
 {\det \left[ \begin{array}{cccc}
d_1 & -a & -\bb^T  & 0 \\
-a & d_2 & -\cc^T & 0 \\ 
-\bb & -\cc & D  & 0 \\
0 & 0 & 0 & E\end{array}\right]}
=\frac{\det \left[ \begin{array}{ccc}
 D & -\x & 0 \\
-\x^T & d_k & -\y^T \\
 0 & -\y & E\end{array}\right]}
 {\det \left[ \begin{array}{ccc}
d_1 & -a & -\bb^T   \\
-a & d_2 & -\cc^T \\ 
-\bb & -\cc & D   \\\end{array}\right]\cdot \det E}.
\end{equation}
Note that we chose to delete row and column $k$ in the denominator rather than either row and column 1 or row and column 2 as in the statement of Lemma 1. This is permissible, however,  because all cofactors of $L_G$ are equal since all row and column sums are zero.  Applying Lemma 1 to vertices 1, and 2 and the subgraph $H$, 
\begin{equation}\label{eq:detfrac2}
r_H(1,2) =\frac{\det \left[ \begin{array}{ccc}
 D & -\x \\
-\x^T & d_k^{\;\prime} \end{array}\right]}
 {\det \left[ \begin{array}{cccc}
d_1 & -a & -\bb^T   \\
-a & d_2 & -\cc^T  \\ 
-\bb & -\cc & D  \end{array}\right]}.
\end{equation}
However, 
\[
\det \left[ \begin{array}{ccc}
 D & -\x & 0 \\
-\x^T & d_k & -\y^T \\
 0 & -\y & E\end{array}\right]
 =\det \left[ \begin{array}{ccc}
 D & -\x & 0 \\
-\x^T & d_k - \mathbf{1}^T\y & -\y^T + \mathbf{1}^TE \\
 0 & -\y & E\end{array}\right].
\]
We have $d_k - \mathbf{1}^T \y = d_k^{\;\prime}$, and, because the column sums of $L_G$ are 0, $-\y^T + \mathbf{1}^TE=0$.

Therefore, 
\[
\det \left[ \begin{array}{ccc}
 D & -\x & 0 \\
-\x^T & d_k & -\y^T \\
 0 & -\y & E\end{array}\right]
 =\det \left[ \begin{array}{ccc}
 D & -\x  \\
-\x^T & d_k^{\;\prime} \end{array}\right]\cdot \det E.
\]
Then, by Equations  \ref{eq:detfrac} and \ref{eq:detfrac2}, $r_G(1,2) = r_H(1,2)$. \\[1mm]
\noindent \textbf{Case 2: } $v$ is one of the vertices $i, j$.  Without loss of generality, take $i = v  =1$ and $j = 2$. Then we may write the Laplacian matrices of $G$ and $H$ as 
\[
L_G = \left[ \begin{array}{cccc}
d_1 & -a & -\bb^T & -\y^T \\
-a & d_2 & -\cc^T  & 0 \\ 
-\bb & -\cc & D  & 0 \\
-\y & 0 & 0  & E
\end{array}\right], \ \ L_H = \left[ \begin{array}{cccc}
d_1^{\;\prime} & -a & -\bb^T  \\
-a & d_2 & -\cc^T  \\ 
-\bb & -\cc & D  \\
\end{array}\right].
\]
where $d_1 = a + \bb^T \mathbf{1} + \y^T\mathbf{1}$ and $d_1^{\;\prime} = a + \bb^T\mathbf{1}$.  Then by Lemma 1 in \cite{BapatWheels},
\[
r_G(1,2) =\frac{\det \left[ \begin{array}{ccc}
 D & 0 \\
 0 & E\end{array}\right]}
 {\det \left[ \begin{array}{cccc}
d_2  & -\cc^T  & 0 \\
 -\cc & D  & 0 \\
0 &  0 & E\end{array}\right]}
=\frac{\det D \det E}
 {\det  \left[ \begin{array}{cccc}
d_2  & -\cc^T  \\
 -\cc & D   \end{array}\right]
\cdot \det E}
=\frac{\det D }
 {\det  \left[ \begin{array}{cccc}
d_2  & -\cc^T  \\
 -\cc & D   \end{array}\right]
}= r_H(1,2),
\]
which completes the proof.
\end{proof}

\subsection{Fibonacci and Lucas Identities}
Both Fibonacci and Lucas numbers are important to the proof of our main results.  By $F_n$ we denote the $n$th Fibonacci number and by $L_m$ we denote the $m$th Lucas number.
\begin{proposition}\label{prop:fneg} $F_{-n}=(-1)^{n+1}F_n.$\end{proposition}
\begin{proof}
For the proof see~\cite{Vajda}.%
\end{proof}
\begin{proposition}\label{prop:fsumsq} $F_{n}^2+F_{n+1}^2 = F_{2n+1}.$\end{proposition}
\begin{proof}
For the proof see~\cite{Vajda}.%
\end{proof}
\begin{proposition}\label{prop:f2m} $F_{2m}=L_mF_m.$\end{proposition}
\begin{proof}
For the proof see~\cite{Vajda}.%
\end{proof}
\begin{proposition}\label{prop:splitsum} $F_{k+m}=F_{k+1}F_m+F_kF_{m-1}$.\end{proposition}
\begin{proof}
For the proof see~\cite{BQ}.%
\end{proof}
\begin{corollary}\label{cor:splitsum} $F_{2m}=F_{m+1}F_m+F_mF_{m-1}.$\end{corollary}
\begin{proposition}\label{prop:splitdiff}$F_{n + m} = F_{n + 1}F_{m + 1} - F_{n - 1}F_{m - 1}$\end{proposition}
\begin{proof}
For the proof see~\cite{BQ}.%
\end{proof}

\begin{proposition}\label{prop:catalan}[Catalan's Identity] $F_{n}^2-F_{n+r}F_{n-r} = (-1)^{n-r}F_r^2$\end{proposition}
\begin{proof}
For the proof see~\cite{Vajda}.%
\end{proof}
\begin{proposition}\label{prop:doca}[d'Ocagne's Identity] $F_nF_{m + 1} - F_mF_{n + 1} = (-1)^mF_{n - m}$.
\end{proposition}
\begin{proof}
For the proof see~\cite{Vajda}.%
\end{proof}
%
%
\begin{proposition}\label{prop:2fm1} $2F_{m+1}=F_m+L_{m}.$\end{proposition}
\begin{proof}
For the proof see~\cite{Vajda}.
\end{proof}
\begin{corollary}\label{cor:2fm1}
$F_{m+1}+F_{m-1} = L_m$.
\end{corollary}
\begin{proposition}\label{prop:lm1} $L_{m+1}=2F_m+F_{m+1}.$\end{proposition}
\begin{proof}
For the proof see~\cite{BQ}.%
\end{proof}
\begin{proposition}\label{prop:5diff} $\displaystyle
5 F_n^2 - L_n^2 = 4 (-1)^n + 1$
\end{proposition}
\begin{proof}
For the proof see~\cite{Vajda}.%
\end{proof}
\begin{corollary}\label{cor:sqr} $\displaystyle \lim_{n\to\infty} \dfrac{L_n}{F_n} = \sqrt{5}.$\end{corollary}
\begin{proposition}\label{prop:15sum} $\displaystyle F_{m} = \frac{1}{5}(L_{m-1} + L_{m+1}).$
\end{proposition}
\begin{proof}
For the proof see~\cite{Vajda}.%
\end{proof}
\begin{proposition}\label{prop:longsum} $F_{2n+2} = 2 F_{2n}  + F_{2n-2} + \cdots + F_2 + 1.$
\end{proposition}

\section{Maximal resistance distance in the straight linear 2-tree}

When determining the resistance distance between node $1$ and node $n$ we use the following algorithm, demonstrated in Figures ~\ref{fig:2treepost}, ~\ref{fig:2treepostnext}, ~\ref{fig:afterkth}, and ~\ref{fig:2treefinalt} 
\begin{itemize}
\item First perform the $\Delta$--Y transform on the leftmost triangle (defined by the vertices $1$, $2$, and $3$).  This results in a new graph node $*$ as shown in Figure~\ref{fig:2treepost}.  
\item Next, sum the weight between vertices $2$ and $*$ with the weight between vertices $2$ and $4$, delete vertex 2 and rename vertex * as vertex 2 as shown in Figure~\ref{fig:2treepostnext}.  
\item Perform a $\Delta$-Y transform on the new left-most triangle.  
\end{itemize}
We continue in this manner until all triangles have been removed and we are left with a pair of parallel edges and a long tail, as in Figure~\ref{fig:2treefinalt}.	 \begin{figure}
\begin{subfigure}[b]{\textwidth}
\begin{center}
\begin{tikzpicture}[line cap=round,line join=round,>=triangle 45,x=1.0cm,y=1.0cm, scale = 1.2]
\draw [line width=.8pt] (-2.,0.)-- (-1.,0.);
\draw [line width=.8pt,dotted] (-1.,0.)-- (0.,0.);
\draw [line width=.8pt] (0.,0.)-- (1.,0.);
\draw [line width=.8pt] (1.,0.)-- (2.,0.);
\draw [line width=.8pt] (2.,0.)-- (1.5,0.866025403784435);
\draw [line width=.8pt] (1.5,0.866025403784435)-- (1.,0.);
\draw [line width=.8pt] (1.,0.)-- (0.5,0.8660254037844366);
\draw [line width=.8pt] (0.5,0.8660254037844366)-- (0.,0.);
\draw [line width=.8pt] (-0.5,0.8660254037844378)-- (-1.,0.);
\draw [line width=.8pt] (-1.,0.)-- (-1.5,0.8660254037844385);
\draw [line width=.8pt] (-1.5,0.8660254037844385)-- (-2.,0.);
\draw [line width=.8pt] (-1.5,0.8660254037844385)-- (-0.5,0.8660254037844378);
\draw [line width=.8pt,dotted] (-0.5,0.8660254037844378)-- (0.5,0.8660254037844366);
\draw [line width=.8pt] (0.5,0.8660254037844366)-- (1.5,0.866025403784435);
\draw [line width=.8pt,dashed,red] (-3.,0.)-- (-2.5,0.28867513459481287);
\draw [line width=.8pt,dashed] (-2.5,0.28867513459481287)-- (-2.,0.);
\draw [line width=.8pt] (-2.5,0.866025403784439)-- (-1.5,0.8660254037844385);
\draw [line width=.8pt,dashed] (-2.5,0.866025403784439)-- (-2.5,0.28867513459481287);
\begin{scriptsize}
\draw [fill=black] (-3.,0.) circle (1.5pt);
\draw[color=black] (-3.025029469669545,-0.22018097717087654) node {$1$};
\draw [fill=black] (-2.,0.) circle (1.5pt);
\draw[color=black] (-1.999600161545304,-0.22571778486972866) node {$3$};
\draw [fill=black] (-1.5,0.8660254037844385) circle (1.5pt);
\draw[color=black] (-1.517959122880888,1.1015212926292562) node {$4$};
\draw [fill=black] (-1.,0.) circle (1.5pt);
\draw[color=black] (-1.0207812765176196,-0.22232820796628497) node {$5$};
\draw [fill=black] (-0.5,0.8660254037844378) circle (1.5pt);
\draw[color=black] (-0.49252981475664714,1.1015212926292562) node {$6$};
\draw [fill=black] (0.,0.) circle (1.5pt);
\draw[color=black] (-0.04196239148993525,-0.22018097717087654) node {$n-4$};
\draw [fill=black] (0.5,0.8660254037844366) circle (1.5pt);
\draw[color=black] (0.4862890702710373,1.100447677231552) node {$n-3$};
\draw [fill=black] (1.,0.) circle (1.5pt);
\draw[color=black] (0.9834669166343056,-0.22018097717087654) node {$n-2$};
\draw [fill=black] (1.5,0.866025403784435) circle (1.5pt);
\draw[color=black] (1.6049392245883911,1.105984484930404) node {$n-1$};
\draw [fill=black] (2.,0.) circle (1.5pt);
\draw[color=black] (1.9778226093608422,-0.222018097717087654) node {$n$};
\draw [fill=black] (-3.,0.) circle (1.5pt);
\draw [fill=black] (-2.5,0.28867513459481287) circle (1.5pt);
\draw[color=black] (-2.7298301233913547,0.41790175387976375) node {$\ast$};
\draw [fill=black] (-2.5,0.866025403784439) circle (1.5pt);
\draw[color=black] (-2.3880203540166076,1.10149108695327) node {$2$};
\end{scriptsize}
\end{tikzpicture}
\end{center}
\caption{A linear 2-tree after the first $\Delta$--Y transformation.}
\label{fig:2treepost}
\end{subfigure}
\vspace{.5cm}

\begin{subfigure}[b]{\textwidth}
\begin{center}

%
%

\begin{tikzpicture}[line cap=round,line join=round,>=triangle 45,x=1.0cm,y=1.0cm, scale = 1.2]
\draw [line width=.8pt] (-2.,0.)-- (-1.,0.);
\draw [line width=.8pt,dotted] (-1.,0.)-- (0.,0.);
\draw [line width=.8pt] (0.,0.)-- (1.,0.);
\draw [line width=.8pt] (1.,0.)-- (2.,0.);
\draw [line width=.8pt] (2.,0.)-- (1.5,0.866025403784435);
\draw [line width=.8pt] (1.5,0.866025403784435)-- (1.,0.);
\draw [line width=.8pt] (1.,0.)-- (0.5,0.8660254037844366);
\draw [line width=.8pt] (0.5,0.8660254037844366)-- (0.,0.);
\draw [line width=.8pt] (-0.5,0.8660254037844378)-- (-1.,0.);
\draw [line width=.8pt] (-1.,0.)-- (-1.5,0.8660254037844385);
\draw [line width=.8pt] (-1.5,0.8660254037844385)-- (-2.,0.);
\draw [line width=.8pt] (-1.5,0.8660254037844385)-- (-0.5,0.8660254037844378);
\draw [line width=.8pt,dotted] (-0.5,0.8660254037844378)-- (0.5,0.8660254037844366);
\draw [line width=.8pt] (0.5,0.8660254037844366)-- (1.5,0.866025403784435);
\draw [line width=.8pt,dashed,red] (-3.5,0.866025403784435)-- (-2.5,0.866025403784435);
\draw [line width=.8pt,dashed] (-2.5,0.866025403784435)-- (-2.,0.);
\draw [line width=.8pt,dashed] (-2.5,0.866025403784439)-- (-1.5,0.8660254037844385);
\begin{scriptsize}
\draw [fill=black] (-3.5,0.866025403784435) circle (1.5pt);
\draw[color=black] (-3.5,1.1) node {$1$};
\draw [fill=black] (-2.,0.) circle (1.5pt);
\draw[color=black] (-1.999600161545304,-0.22571778486972866) node {$3$};
\draw [fill=black] (-1.5,0.8660254037844385) circle (1.5pt);
\draw[color=black] (-1.517959122880888,1.1015212926292562) node {$4$};
\draw [fill=black] (-1.,0.) circle (1.5pt);
\draw[color=black] (-1.0207812765176196,-0.22232820796628497) node {$5$};
\draw [fill=black] (-0.5,0.8660254037844378) circle (1.5pt);
\draw[color=black] (-0.49252981475664714,1.1015212926292562) node {$6$};
\draw [fill=black] (0.,0.) circle (1.5pt);
\draw[color=black] (-0.04196239148993525,-0.22018097717087654) node {$n-4$};
\draw [fill=black] (0.5,0.8660254037844366) circle (1.5pt);
\draw[color=black] (0.4862890702710373,1.100447677231552) node {$n-3$};
\draw [fill=black] (1.,0.) circle (1.5pt);
\draw[color=black] (0.9834669166343056,-0.22018097717087654) node {$n-2$};
\draw [fill=black] (1.5,0.866025403784435) circle (1.5pt);
\draw[color=black] (1.6049392245883911,1.105984484930404) node {$n-1$};
\draw [fill=black] (2.,0.) circle (1.5pt);
\draw[color=black] (1.9778226093608422,-0.222018097717087654) node {$n$};
\draw [fill=black] (-2.5,0.866025403784439) circle (1.5pt);
\draw[color=black] (-2.3880203540166076,1.10149108695327) node {$2$};
\end{scriptsize}
\end{tikzpicture}
%
%
%
\end{center}
\caption{A linear 2-tree after the first $\Delta$--Y transformation once two edges are merged, a vertex is removed, and a vertex is renamed.  }
\label{fig:2treepostnext}
\end{subfigure}
\caption{A linear 2-tree after step one and step two of the algorithm . The dashed edges are the edges with new weights after the transformation.  The dashed red edge (from node 1 to node 2) is the ``tail'' of the transformation.}
\label{fig:combfigures}
\end{figure}
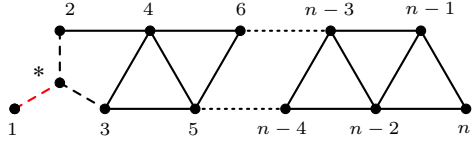
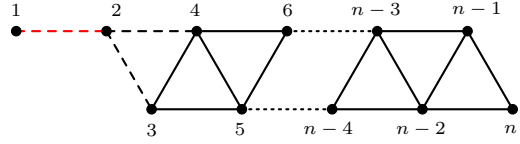

\begin{remark}\label{rem26}
 Notice that the $k$th $\Delta$--Y transformation transforms a triangle with node labels $k,k +1$,  and $ k+2$ into a $Y$ with nodes labeled $k, k +1, k+2$, and $\ast$. We adopt the convention that the edge $R_A^k = r(k,k+1)$, $R_B^k = r(k,k+2)$, and $R_C^k = r(k+1,k+2)$.  Thus, in the subsequent, equivalent network, $t_k = R_3^k = r( \ast, k)$, $s_k = R_2^k = r( \ast, k+1)$, and $b_k = R_1^k = r( \ast, k+2)$.  
We call $t_k$ the tail resistance because after a sequence of $\Delta$--Y transforms several resistors are left in a tail.  This resistance will never be involved in another $\Delta$-Y transformation. Notice that when performing the merge step described in the second bullet above, $s_k$ will always be on the edge that is merged. It is easy to verify that $R_C^k=1$  for every $k$.  See Figure~\ref{fig:afterkth}.
\end{remark}
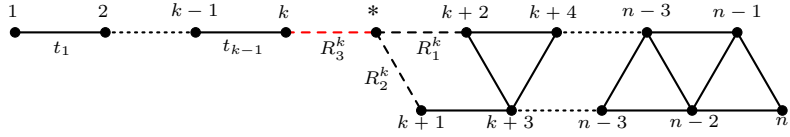
\begin{figure}[h!]
\begin{center}
\begin{tikzpicture}[line cap=round,line join=round,>=triangle 45,x=1.0cm,y=1.0cm,scale = 1.2]
\draw [line width=.8pt] (-3.,2.)-- (-2.,2.);
\draw [line width=.8pt,dotted] (-2.,2.)-- (-1.,2.);
\draw [line width=.8pt] (-1.,2.)-- (0.,2.);
\draw [line width=.8pt,dashed,red] (0.,2.)-- (1.,2.);
\draw [line width=.8pt,dashed] (1.,2.)-- (2.,2.);
\draw [line width=.8pt,dashed] (1.,2.)-- (1.5,1.1339745962155612);
\draw [line width=.8pt] (2.,2.)-- (2.5,1.1339745962155607);
\draw [line width=.8pt] (2.5,1.1339745962155607)-- (1.5,1.1339745962155612);
\draw [line width=.8pt] (2.,2.)-- (3.,2.);
\draw [line width=.8pt] (3.,2.)-- (2.5,1.1339745962155607);
\draw [line width=.8pt,dotted] (2.5,1.1339745962155607)-- (3.5,1.1339745962155598);
\draw [line width=.8pt,dotted] (3.,2.)-- (4.,2.);
\draw [line width=.8pt] (4.,2.)-- (3.5,1.1339745962155598);
\draw [line width=.8pt] (3.5,1.1339745962155598)-- (4.5,1.1339745962155583);
\draw [line width=.8pt] (4.5,1.1339745962155583)-- (4.,2.);
\draw [line width=.8pt] (4.,2.)-- (5.,2.);
\draw [line width=.8pt] (5.,2.)-- (4.5,1.1339745962155583);
\draw [line width=.8pt] (4.5,1.1339745962155583)-- (5.5,1.1339745962155563);
\draw [line width=.8pt] (5.5,1.1339745962155563)-- (5.,2.);
\begin{scriptsize}
\draw [fill=black] (-3.,2.) circle (1.5pt);
\draw[color=black] (-3.0254981493173747,2.2417294054505863) node {$1$};
\draw [fill=black] (-2.,2.) circle (1.5pt);
\draw[color=black] (-2.0254981493173747,2.2417294054505863) node {$2$};
\draw [fill=black] (-1.,2.) circle (1.5pt);
\draw[color=black] (-1.0097018789148775,2.2417294054505863) node {$k-1$};
\draw[color=black] (-2.4752751695160222,1.8242930709498067) node {$t_1$};
\draw [fill=black] (0.,2.) circle (1.5pt);
\draw[color=black] (-0.025649241962458255,2.23114819430056) node {$k$};
\draw[color=black] (-0.4700601102635509,1.8242930709498067) node {$t_{k-1}$};
\draw [fill=black] (1.,2.) circle (1.5pt);
\draw[color=black] (0.968984606139987,2.2417294054505863) node {$\ast$};
\draw[color=black] (0.5245737378388945,1.8242930709498067) node {$R_3^k$};
\draw [fill=black] (2.,2.) circle (1.5pt);
\draw[color=black] (1.9847808765424844,2.2205669831505346) node {$k+2$};
\draw[color=black] (1.5874639524912618,1.8242930709498067) node {$R_1^k$};
\draw [fill=black] (1.5,1.1339745962155612) circle (1.5pt);
\draw[color=black] (1.4980451636412877,1.0037277008975454) node {$k+1$};
\draw [fill=black] (2.5,1.1339745962155607) circle (1.5pt);
\draw[color=black] (2.4820978005937073,1.0037277008975454) node {$k+3$};
\draw [fill=black] (3.,2.) circle (1.5pt);
\draw[color=black] (2.958252302344878,2.2205669831505346) node {$k+4$};
\draw [fill=black] (3.5,1.1339745962155598) circle (1.5pt);
\draw[color=black] (3.5084752821462306,0.9931464897475193) node {$n-3$};
\draw [fill=black] (4.,2.) circle (1.5pt);
\draw[color=black] (3.9952109950474273,2.23114819430056) node {$n-3$};
\draw [fill=black] (4.5,1.1339745962155583) circle (1.5pt);
\draw[color=black] (4.524271552548727,1.0037277008975454) node {$n-2$};
\draw [fill=black] (5.,2.) circle (1.5pt);
\draw[color=black] (5.000426054299898,2.2205669831505346) node {$n-1$};
\draw [fill=black] (5.5,1.1339745962155563) circle (1.5pt);
\draw[color=black] (5.497742978351121,1.0143089120475712) node {$n$};
\draw[color=black] (1.01118309566653382,1.4909849197239878) node {$R_2^k$};
\end{scriptsize}
\end{tikzpicture}
\end{center}
\caption{A straight linear 2-tree after the $k$th $\Delta-Y$ transformation.  The dashed edges are the edges with new weights after the transformation. The dashed red edge (from node $k$ to node $\ast$) is the newest ``tail'' of the transformation.}
\label{fig:afterkth}
\end{figure}

For the remainder of this section we assume that all edge weights are one.

\begin{lemma}
When performing the $k$th $\Delta$--Y transformation, $R_A^k+R_B^k+R_C^k = \dfrac{F_{2k+2}}{F_{2k}}.$
\end{lemma}
\begin{proof}
We proceed by induction. When performing the first $\Delta$--Y transformation, $R_A^1= R_B^1= R_C^1 =1.$  Hence $R_A^1+R_B^1+R_C^1 = 3= 3/1 = F_4/F_2$ as expected.  We now assume that upon performing the $k-1$st transformation $R_A^{k-1}+R_B^{k-1}+R_C^{k-1} = \dfrac{F_{2k}}{F_{2k-2}}$ and will show that $R_A^k+R_B^k+R_C^k = \dfrac{F_{2k+2}}{F_{2k}}.$ By Definition~\ref{def:dy} the $k-1$st $\Delta$--Y transformation yields
\[R_1^{k-1}=\frac{R_B^{k-1}R_C^{k-1}}{R_A^{k-1}+R_B^{k-1}+R_C^{k-1}} = \frac{R_B^{k-1}}{R_A^{k-1}+R_B^{k-1}+R_C^{k-1}}\]
and\[R_2^{k-1}=\frac{R_A^{k-1}R_C^{k-1}}{R_A^{k-1}+R_B^{k-1}+R_C^{k-1}} = \frac{R_A^{k-1}}{R_A^{k-1}+R_B^{k-1}+R_C^{k-1}}.\]
Hence
\[\begin{array}{r}
R_A^{k}+R_B^{k}+R_C^{k} = R_1^{k-1} +(R_2^{k-1} + 1)+R_C^k = 2 + \ds{\frac{R_A^{k-1}+R_B^{k-1}}{R_A^{k-1}+R_B^{k-1}+R_C^{k-1}}}\\[4mm]
 \ds{=3 - \frac{R_C^{k-1}}{R_A^{k-1}+R_B^{k-1}+R_C^{k-1}}
=3 - \frac{1}{R_A^{k-1}+R_B^{k-1}+R_C^{k-1}}}\\[6mm]
\ds{=3 - \frac{F_{2k-2}}{F_{2k}}
= \frac{3F_{2k}-F_{2k-2}}{F_{2k}}}
\ds{
=\frac{F_{2k+2}}{F_{2k}}}\end{array}.
\]
\end{proof}
\begin{lemma}\label{lem3} Following the algorithm detailed for application of the $\Delta$--Y transform to the linear 2-tree with $m$ cells, for $1 \leq p \leq m-1$ after the $p$th 
$\Delta$--Y transform 
\[
{s_{p}} = \dfrac{F_{p}^2}{F_{2p+2}}, \ 
{b_p}=\dfrac{F_{p+1}}{L_{p+1}}, \text{ and }{t_p}=\dfrac{F_pF_{p+1}}{L_pL_{p+1}}.
\]
where $s_p$, $b_p$ and $t_p$ are as defined in Remark~\ref{rem26}.
\end{lemma}
\begin{proof}  We show this result inductively.  After the first $\Delta$--Y transform ${s_1} = 1/3 = F_1^2/F_4$, ${b_1} = 1/3 = F_2/L_2$, and ${t_1} = 1/3= (F_1F_2)/(L_1L_2).$

Suppose after the $k-1$st $\Delta$--Y transformation ${s_{k-1}} = \dfrac{F_{k-1}^2}{F_{2k}}$, ${b_{k-1}}=\dfrac{F_k}{L_k}$.  
Let $R_A = {b_{k-1}}$, $R_B = {s_{k-1}}+1$, and $R_C=1$.  Then by Definition~\ref{def:dy} we have
    \[\begin{array}{c}
   \ds{ 
   R_1 = \frac{R_BR_C}{R_A+R_B+R_C}
     = \frac{{s_{k-1}}+1}{R_A+R_B+R_C}
     =\frac{1+\frac{F_{k-1}^2}{F_{2k}}}{\frac{F_{2k+2}}{F_{2k}}} \hspace*{3cm}
     }\\[6mm]
    \ds{
     =\frac{F_{2k}+F_{k-1}^2}{F_{2k+2}}
     =\frac{F_{k}L_k+F_{k-1}^2}{F_{k+1}L_{k+1}}
     =\frac{F_{k}F_{k+1}+F_kF_{k-1}+F_{k-1}^2}{F_{k+1}L_{k+1}}
     }\\[6mm]
     \ds{
     =\frac{F_{k+1}\left(F_{k}+F_{k-1}\right)}{F_{k+1}L_{k+1}}
     =\frac{F_{k+1}}{L_{k+1}}
     ={b_{k}}
     }.\end{array}\]
     Similarly,
    \[\begin{array}{r}
    \ds{
    R_2 = \frac{R_AR_C}{R_A+R_B+R_C}
     =\frac{\frac{F_{k}}{L_k}}{\frac{F_{2k+2}}{F_{2k}}}
     =\frac{F_{2k}F_{k}}{F_{2k+2}L_k}
     =\frac{F_{k}^2}{F_{2k+2}}
     ={s_k}.
    }\\
     \end{array}\]
     Also, we have
    \[\begin{array}{r}
    \ds{
    R_3 = \frac{R_AR_B}{R_A+R_B+R_C}
    =R_1b_{k-1}
    =\dfrac{F_kF_{k+1}}{L_kL_{k+1}}
    ={t_k}
    }.\end{array}\]
    \end{proof}

	\begin{theorem}\label{thm:effRes}
	Let $G_n$ be the linear 2-tree with $n$ vertices and $m=n-2$ cells. Then the resistance distance between nodes $1$ and $n$ is given by
	\[r(1,n)=\frac{2F_{m+1}^2}{L_{m+1}L_m}+\sum_{i=1}^{m-1}\frac{F_iF_{i+1}}{L_iL_{i+1}}
	 = \frac{m+1}{5} + \frac{4F_{m+1}}{5L_{m+1}}.\]
	\end{theorem}
	\begin{proof}
	We determine the resistance distance by performing $m-1$ $\Delta$--Y transformations.  After performing these transformations we are left with a circuit with a long series tail, and  two nodes connected in parallel (as shown in Figure~\ref{fig:2treefinalt}).  We first perform a parallel transformation on these two edges, specifically
	\[\begin{array}{c}
	\ds{
	\left(\frac{1}{1+{s_{m-1}}}+\frac{1}{1+{b_{m-1}}}\right)^{-1} = \left(\frac{1}{1+\frac{F_{m-1}^2}{F_{2m}}}+\frac{1}{1+\frac{F_m}{L_m}}\right)^{-1} 
	= \left(\frac{F_{2m}}{F_{2m}+{F_{m-1}^2}}+\frac{L_m}{L_m+{F_m}}\right)^{-1}  \hspace*{2cm}
	} \\[6mm]
	\ds{
	= \left(\frac{F_{2m}}{F_{m}F_{m+1}+F_mF_{m-1}+{F_{m-1}^2}}+\frac{L_m}{2F_{m+1}}\right)^{-1}  
	= \left(\frac{F_{2m}}{F_{m}F_{m+1}+F_{m-1}F_{m+1}}+\frac{L_m}{2F_{m+1}}\right)^{-1}  
	} \\[6mm]
	\ds{
	= \left(\frac{F_{2m}}{F_{m+1}^2}+\frac{L_m}{2F_{m+1}}\right)^{-1}  
	= \left(\frac{2F_{2m}+F_{m+1}L_m}{2F_{m+1}^2}\right)^{-1}  
	= \left(\frac{L_m(2F_m+F_{m+1})}{2F_{m+1}^2}\right)^{-1}  
	} \\[6mm]
	\ds{
	= \left(\frac{L_mL_{m+1}}{2F_{m+1}^2}\right)^{-1}  	
	=\frac{2F^2_{m+1}}{L_{m+1}L_m}
	}.\end{array}\]
 All that remains is to consider the series portion of the circuit which contains the ``tails'' of each transformation.  By Lemma~\ref{lem3} the tail after transformation $i$ is $\dfrac{F_iF_{i+1}}{L_iL_{i+1}}$.  Summing the resistances yields the desired result.
	For the proof of the second equality, see Propositions \ref{prop:rm1nA} and  \ref{prop:rm1nB}.
\end{proof}

	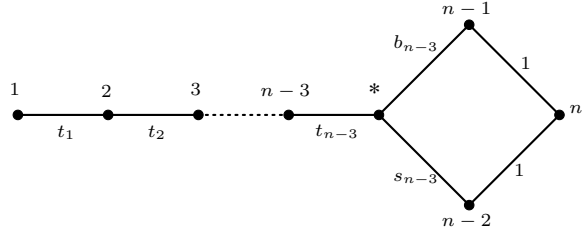
\begin{figure}
\begin{center}

\begin{tikzpicture}[line cap=round,line join=round,>=triangle 45,x=1.0cm,y=1.0cm, scale = 1.2]
\draw [line width=.8pt] (-3.,2.)-- (-2.,2.);
\draw [line width=.8pt] (-2.,2.)-- (-1.,2.);
\draw [line width=.8pt,dotted] (-1.,2.)-- (0.,2.);
\draw [line width=.8pt] (0.,2.)-- (1.,2.);
\draw [line width=.8pt] (1.,2.)-- (2.,3.);
\draw [line width=.8pt] (2.,3.)-- (3.,2.);
\draw [line width=.8pt] (3.,2.)-- (2.,1.);
\draw [line width=.8pt] (2.,1.)-- (1.,2.);
\begin{scriptsize}
\draw [fill=black] (-3.,2.) circle (1.5pt);
\draw[color=black] (-3.0357596415464108,2.283574198210263) node {$1$};
\draw [fill=black] (-2.,2.) circle (1.5pt);
\draw[color=black] (-2.0176831175804275,2.269627944457304) node {$2$};
\draw[color=black] (-2.463963237675105,1.8070253468803786) node {$t_1$};
\draw [fill=black] (-1.,2.) circle (1.5pt);
\draw[color=black] (-1.0274991011203625,2.283574198210263) node {$3$};
\draw[color=black] (-1.4598329674620811,1.8070253468803786) node {$t_2$};
\draw [fill=black] (0.,2.) circle (1.5pt);
\draw[color=black] (-0.037315084660296774,2.269627944457304) node {$n-3$};
\draw [fill=black] (1.,2.) circle (1.5pt);
\draw[color=black] (0.9528689317997687,2.297520451963222) node {$\ast$};
\draw[color=black] (0.5344813192110087,1.8070253468803786) node {$t_{n-3}$};
\draw [fill=black] (2.,3.) circle (1.5pt);
\draw[color=black] (1.9709454557657518,3.1737582146703295) node {$n-1$};
\draw[color=black] (1.4107190819181155,2.7856393333167757) node {$b_{n-3}$};
\draw [fill=black] (3.,2.) circle (1.5pt);
\draw[color=black] (3.1842695322731562,2.060434138162924) node {$n$};
\draw[color=black] (2.6264193821548094,2.5764455270223956) node {$1$};
\draw [fill=black] (2.,1.) circle (1.5pt);
\draw[color=black] (1.9709454557657518,0.8192175541496014) node {$n-2$};
\draw[color=black] (2.556688113390016,1.3910139580209073) node {$1$};
\draw[color=black] (1.409878131533223,1.286312145051499) node {$s_{n-3}$};
\end{scriptsize}
\end{tikzpicture}

%
%
\end{center}
\caption{A linear 2-tree after the $m-1$st $\Delta$--Y transformation.}
\label{fig:2treefinalt}
\end{figure}
	The proof that this is the maximal resistance distance for the straight linear 2-tree with $n$ vertices will be given in Section~\ref{sec:mono}.  However we show one more interesting result.
	\begin{theorem}\label{cor:onefifth}
	Let $G$ be the straight linear 2-tree with $n$ vertices and $H$ be the straight linear 2-tree with $n+1$ vertices.
	Then 
	\[\lim_{n\rightarrow \infty} \left[r_{H} (1, n+1) - r_G(1,n)\right] = \frac{1}{5}.\]
	\end{theorem}
	\begin{proof}
	Let $m = n-2$.  By Theorem~\ref{thm:effRes} 
	\[r_H(1,n+1) =  \frac{m+2}{5} + \frac{4F_{m+2}}{5L_{m+2}}
	\text{ and }
	r_G(1,n) =  \frac{m+1}{5} + \frac{4F_{m+1}}{5L_{m+1}}.
	\]
	Subtracting yields
	\[r_H(1,n+1)-r_G(1,n)=
	\frac{1}{5} + \frac{4F_{m+2}}{5L_{m+2}} - \frac{4F_{m+1}}{5L_{m+1}}.
	\]
	 Taking limits and applying Corollary~\ref{cor:sqr} gives
	 \[
	 \lim_{n\rightarrow\infty}\left[r_H(1,n+1)-r_G(1,n)\right] =
	  \frac{1}{5} + \frac{4}{5}\left(\lim_{n\rightarrow\infty}\frac{F_{m+2}}{L_{m+2}}-\lim_{n\rightarrow\infty}\frac{F_{m+1}}{L_{m+1}}\right)= \frac{1}{5} + \frac{4}{5}\left(\frac{1}{\sqrt{5}}-\frac{1}{\sqrt{5}}\right)=\frac{1}{5}.\]

	 \end{proof}
	\section{ Resistance between arbitrary points on a straight linear 2-tree}
	We now extend the $\Delta$--Y transform method of the previous section to derive the resistance distance between any two nodes in the straight linear 2-tree.
	
	\begin{lemma}\label{lem3g} Let $G$ be the straight linear 2-tree with $m$ cells and resistance equal to one on all edges except the edge between nodes 1 and 2 for which the resistance is $F_{2p+1}/F_{2p+2}$ where $p$ is an integer greater than or equal to 0.  Following the algorithm detailed in Section~\ref{sec:note} after the $i$th (for $1 \leq i \leq m-1$) 
$\Delta$--Y transform 
\[
{s_{i,p}} = \dfrac{F_{i}F_{i+2p}}{F_{2i+2p+2}},\  
{b_{i,p}}=\dfrac{F_{i+1}F_{i+2p+1}}{F_{2i+2p+2}}, \text{ and }{t_{i,p}}=\dfrac{F_{i}F_{i+1}F_{i+2p}F_{i+2p+1}}{F_{2i+2p}F_{2i+2p+2}}.
\]
\end{lemma}
\begin{proof}  We show this result inductively.  After the first $\Delta$--Y transform 
\[\begin{aligned}{s_{1,p}} &= F_{2p+1}/F_{2p+4} = (F_1F_{1+2p})/(F_{2+2p+2}),\\
{b_{1,p}} &= F_{2p+2}/F_{2p+4} = (F_{1+1}F_{1+2p+1})/F_{2+2p+2}, \text{ and}\\
{t_{1,p}} &= F_{2p+1}/F_{2p+4}=(F_{1}F_{2}F_{1+2p}F_{1+2p+1}/{F_{2+2p}F_{2+2p+2}}.\end{aligned}\]
Suppose after the $k-1$st $\Delta$--Y transformation ${s_{k-1,p}} = \dfrac{F_{k-1}F_{k-1+2p}}{F_{2k+2p}}$ and ${b_{k-1,p}}=\dfrac{F_{k}F_{k+2p}}{F_{2k+2p}}.$  
Then $R_A^k = {b_{k-1,p}}$ and $R_B^k = {s_{k-1,p}}+1$, and $R_C^k=1$ so
\[\begin{array}{c}
\ds{ R_A+R_B+R_C = {b_{k-1,p}}+ {s_{k-1,p}}+1+1
=\frac{F_{k}F_{k+2p}+F_{k-1}F_{k-1+2p}+2F_{2k+2p}}{F_{2k+2p}}}
\\[6mm]
\ds{
=\frac{F_{2k+2p-1}+2F_{2k+2p}}{F_{2k+2p}}
=\frac{F_{2k+2p+2}}{F_{2k+2p}}}.\end{array}\]
Also we have
    \[\begin{array}{c}
    \ds{R_1 = \frac{R_BR_C}{R_A+R_B+R_C}
     = \frac{{s_{k-1,p}}+1}{R_A+R_B+R_C}
     =\frac{1+\frac{F_{k-1}F_{k-1+2p}}{F_{2k+2p}}}{\frac{F_{2k+2p+2}}{F_{2k+2p}}} \hspace*{3cm}
     }
     \\[6mm]
     \ds{
     =\frac{F_{2k+2p}+F_{k-1}F_{k-1+2p}}{F_{2k+2p+2}}
     =\frac{F_{k+1}F_{k+2p}+F_kF_{k+2p-1}+F_{k-1}F_{k-1+2p}}{F_{2k+2p+2}}
     }
     \\[6mm]
     \ds{
     =\frac{F_{k+1}F_{k+2p}+F_{k+1}F_{k-1+2p}}{F_{2k+2p+2}}
     =\frac{F_{k+1}F_{k+2p+1}}{F_{2k+2p+2}}
     ={b_{k,p}}}.\end{array}\]
    We also have
    \[
        R_2 = \frac{R_AR_C}{R_A+R_B+R_C}
     =\frac{\frac{F_{k}F_{k+2p}}{F_{2k+2p}}}{\frac{F_{2k+2p+2}}{F_{2k+2p}}}
   =\frac{{F_{k}F_{k+2p}}}{{F_{2k+2p+2}}}
     ={s_{k,p}},
     \]
and
    \[
    R_3 = \frac{R_AR_B}{R_A+R_B+R_C}
    =R_1{b_{k-1,p}}
    =\dfrac{F_{k+1}F_{k+2p+1}F_{k}F_{k+2p}}{F_{2k+2p+2}F_{2k+2p}}
    ={t_{k,p}}.
    \]
    \end{proof}
\begin{lemma}\label{lem:splusb} With reference to Lemma~\ref{lem3} and Lemma~\ref{lem3g},
\[
{s_p}+{b_p} = \frac{F_{2p+1}}{F_{2p+2}}, \text{ and } 
{s_{i,p}}+{b_{i,p}} = \frac{F_{2i + 2p+1}}{F_{2i + 2p+2}}.
\]

\end{lemma}
\begin{proof}  Using Lemmas~\ref{lem3} and~\ref{lem3g}, the proofs of both equalities are straightforward.
\end{proof}

	 \begin{lemma}\label{lem43} Given the straight linear 2-tree with $n$ vertices and $m = n-2$ cells, the resistance distance between node $j$ and node $j+k$ where $j \geq 1$ and $2 \leq k \leq n-j$ is
	 \begin{equation}\label{eq:emilyformulagen}
	 \begin{array}{l}
	\ds{ \sum_{i=1}^{k-1} \dfrac{F_{i}F_{i+1}F_{i+2j-2}F_{i+2j-1}}{F_{2i+2j-2}F_{2i+2j}} + 
	\left(\frac{1}{b_{k-1,j-1}} +\frac{1}{1+s_{k-1,j-1}+s_{m-j-k+1}+b_{m-j-k+1} }\right)^{-1}}\\[6mm]
	\ds{=\sum_{i=1}^{k-1} \dfrac{F_{i}F_{i+1}F_{i+2j-2}F_{i+2j-1}}{F_{2i+2j-2}F_{2i+2j}} + 
		\frac{(F_{2m-2j-2k+4}F_{k-1}F_{2j+k-3}+F_{2m-2j-2k+5}F_{2j+2k-2})F_{k}F_{2j+k-2}}
		{F_{2j+2k-2}F_{2m+2} }, 
	}
	 \end{array}\end{equation}
	 where ${s_{i,p}}$ and ${b_{i,p}}$ are as given in Lemma~\ref{lem3g}.
	 \end{lemma}
	 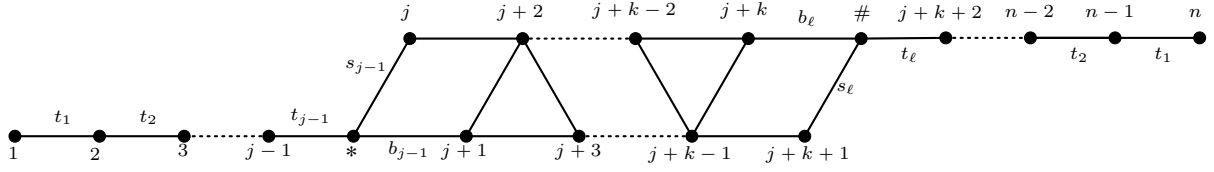
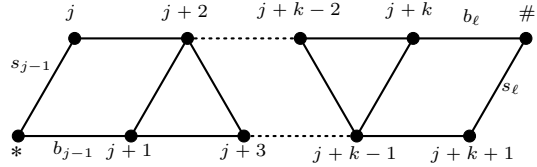
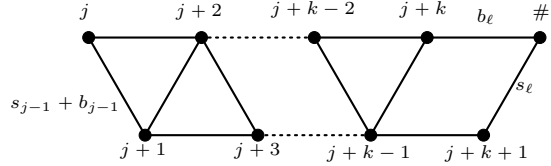
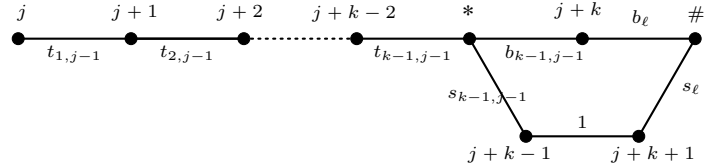
\begin{figure}

\begin{subfigure}[b]{\textwidth}
\begin{center}
\begin{tikzpicture}[line cap=round,line join=round,>=triangle 45,x=1.0cm,y=1.0cm,scale = 1.5]
\draw [line width=.8pt] (0.,3.)-- (-0.75,3.);
\draw [line width=.8pt,dotted] (-0.75,3.)-- (-1.5,3.);
\draw [line width=.8pt] (-1.5,3.)-- (-2.25,3.);
\draw [line width=.8pt] (-2.25,3.)-- (-3.,3.);
\draw [line width=.8pt] (4.5,3.8660254037844353)-- (5.25,3.87228);
\draw [line width=.8pt,dotted] (5.25,3.87228)-- (6.,3.87228);
\draw [line width=.8pt] (6.,3.87228)-- (6.75,3.87228);
\draw [line width=.8pt] (6.75,3.87228)-- (7.5,3.87228);
\draw [line width=.8pt] (6.,3.87228)-- (6.75,3.87228);
\draw [line width=.8pt] (0.,3.)-- (0.5,3.8660254037844393);
\draw [line width=.8pt] (1.,3.)-- (1.5,3.866025403784439);
\draw [line width=.8pt] (1.5,3.866025403784439)-- (2.,3.);
\draw [line width=.8pt] (2.5,3.8660254037844384)-- (3.,3.);
\draw [line width=.8pt] (3.,3.)-- (3.5,3.866025403784437);
\draw [line width=.8pt] (4.,3.)-- (4.5,3.8660254037844353);
\draw [line width=.8pt] (4.5,3.8660254037844353)-- (3.5,3.866025403784437);
\draw [line width=.8pt] (3.5,3.866025403784437)-- (2.5,3.8660254037844384);
\draw [line width=.8pt,dotted] (2.5,3.8660254037844384)-- (1.5,3.866025403784439);
\draw [line width=.8pt] (1.5,3.866025403784439)-- (0.5,3.8660254037844393);
\draw [line width=.8pt] (0.,3.)-- (1.,3.);
\draw [line width=.8pt] (1.,3.)-- (2.,3.);
\draw [line width=.8pt,dotted] (2.,3.)-- (3.,3.);
\draw [line width=.8pt] (3.,3.)-- (4.,3.);
\begin{scriptsize}
\draw [fill=black] (0.,3.) circle (1.5pt);
\draw[color=black] (-0.016037339671645646,2.864228751505624) node {$\ast$};
\draw [fill=black] (1.,3.) circle (1.5pt);
\draw[color=black] (0.9832460507492192,2.864228751505624) node {$j+1$};
\draw [fill=black] (0.5,3.8660254037844393) circle (1.5pt);
\draw[color=black] (0.46637533156601324,4.093232461563467) node {$j$};
\draw [fill=black] (1.5,3.866025403784439) circle (1.5pt);
\draw[color=black] (1.477144737968727,4.093232461563467) node {$j+2$};
\draw [fill=black] (2.,3.) circle (1.5pt);
\draw[color=black] (1.9940154571519328,2.864228751505624) node {$j+3$};
\draw [fill=black] (2.5,3.8660254037844384) circle (1.5pt);
\draw[color=black] (2.487914144371441,4.116204493527165) node {$j+k-2$};
\draw [fill=black] (3.,3.) circle (1.5pt);
\draw[color=black] (2.9818128315909487,2.841256719541926) node {$j+k-1$};
\draw [fill=black] (3.5,3.866025403784437) circle (1.5pt);
\draw[color=black] (3.4642255028286075,4.116204493527165) node {$j+k$};
\draw [fill=black] (4.,3.) circle (1.5pt);
\draw[color=black] (4.02704028593921,2.841256719541926) node {$j+k+1$};
\draw [fill=black] (4.5,3.8660254037844353) circle (1.5pt);
\draw[color=black] (4.509452957176869,4.090662541472712) node {$\#$};
\draw [fill=black] (-0.75,3.) circle (1.5pt);
\draw[color=black] (-0.751142362509983,2.864228751505624) node {$j-1$};
\draw[color=black] (-0.3778468430998898,3.169010286897566) node {$t_{j-1}$};
\draw [fill=black] (-1.5,3.) circle (1.5pt);
\draw[color=black] (-1.5092194173120184,2.875714767487473) node {$3$};
\draw [fill=black] (-2.25,3.) circle (1.5pt);
\draw[color=black] (-2.2902685040777517,2.841256719541926) node {$2$};
\draw [fill=black] (-3.,3.) circle (1.5pt);
\draw[color=black] (-3.0138875109342402,2.8527427355237753) node {$1$};
\draw[color=black] (-2.5716758956330525,3.16792638278866) node {$t_1$};
\draw[color=black] (-1.8135988408310175,3.16792638278866) node {$t_2$};
\draw [fill=black] (5.25,3.87228) circle (1.5pt);
\draw[color=black] (5.19861391608781,4.093232461563467) node {$j+k+2$};
\draw[color=black] (4.928692540514358,3.72118250700081666) node {$t_\ell $};
\draw [fill=black] (6.,3.87228) circle (1.5pt);
\draw[color=black] (5.991149018835392,4.104718477545316) node {$n-2$};
\draw [fill=black] (6.75,3.87228) circle (1.5pt);
\draw[color=black] (6.703282009710032,4.104718477545316) node {$n-1$};
\draw [fill=black] (7.5,3.87228) circle (1.5pt);
\draw[color=black] (7.461359064512067,4.093232461563467) node {$n$};
\draw[color=black] (7.179951672956766,3.7233110859900155) node {$t_1$};
\draw[color=black] (6.43336063413658,3.7233110859900155) node {$t_2$};
\draw[color=black] (0.11605184411961811,3.6165627983167337) node {$s_{j-1}$};
\draw[color=black] (4.365877757403756,3.4213005266253007) node {$s_\ell$};
\draw[color=black] (4.021297277948285,4.051947501518089) node {$b_\ell$};
\draw[color=black] (0.4917162276839946,2.8733459033331897) node {$b_{j-1}$};
\end{scriptsize}
\end{tikzpicture}
\end{center}
\caption{The straight linear 2-tree $G$ after $j-1$ $\Delta$--Y transformations on the left and $m-j-k+1$ transformations on the right.}
\label{sf:step1}
\end{subfigure}
\vspace{.5cm}

\begin{subfigure}[b]{\textwidth}
\begin{center}
\begin{tikzpicture}[line cap=round,line join=round,>=triangle 45,x=1.0cm,y=1.0cm,scale = 1.5]
\draw [line width=.8pt] (0.,3.)-- (0.5,3.8660254037844393);
\draw [line width=.8pt] (1.,3.)-- (1.5,3.866025403784439);
\draw [line width=.8pt] (1.5,3.866025403784439)-- (2.,3.);
\draw [line width=.8pt] (2.5,3.8660254037844384)-- (3.,3.);
\draw [line width=.8pt] (3.,3.)-- (3.5,3.866025403784437);
\draw [line width=.8pt] (4.,3.)-- (4.5,3.8660254037844353);
\draw [line width=.8pt] (4.5,3.8660254037844353)-- (3.5,3.866025403784437);
\draw [line width=.8pt] (3.5,3.866025403784437)-- (2.5,3.8660254037844384);
\draw [line width=.8pt,dotted] (2.5,3.8660254037844384)-- (1.5,3.866025403784439);
\draw [line width=.8pt] (1.5,3.866025403784439)-- (0.5,3.8660254037844393);
\draw [line width=.8pt] (0.,3.)-- (1.,3.);
\draw [line width=.8pt] (1.,3.)-- (2.,3.);
\draw [line width=.8pt,dotted] (2.,3.)-- (3.,3.);
\draw [line width=.8pt] (3.,3.)-- (4.,3.);
\begin{scriptsize}
\draw [fill=black] (0.,3.) circle (1.5pt);
\draw[color=black] (-0.016037339671645646,2.864228751505624) node {$\ast$};
\draw [fill=black] (1.,3.) circle (1.5pt);
\draw[color=black] (0.9832460507492192,2.864228751505624) node {$j+1$};
\draw [fill=black] (0.5,3.8660254037844393) circle (1.5pt);
\draw[color=black] (0.46637533156601324,4.093232461563467) node {$j$};
\draw [fill=black] (1.5,3.866025403784439) circle (1.5pt);
\draw[color=black] (1.477144737968727,4.093232461563467) node {$j+2$};
\draw [fill=black] (2.,3.) circle (1.5pt);
\draw[color=black] (1.9940154571519328,2.864228751505624) node {$j+3$};
\draw [fill=black] (2.5,3.8660254037844384) circle (1.5pt);
\draw[color=black] (2.487914144371441,4.116204493527165) node {$j+k-2$};
\draw [fill=black] (3.,3.) circle (1.5pt);
\draw[color=black] (2.9818128315909487,2.841256719541926) node {$j+k-1$};
\draw [fill=black] (3.5,3.866025403784437) circle (1.5pt);
\draw[color=black] (3.4642255028286075,4.116204493527165) node {$j+k$};
\draw [fill=black] (4.,3.) circle (1.5pt);
\draw[color=black] (4.02704028593921,2.841256719541926) node {$j+k+1$};
\draw [fill=black] (4.5,3.8660254037844353) circle (1.5pt);
\draw[color=black] (4.509452957176869,4.090662541472712) node {$\#$};
\draw[color=black] (0.11605184411961811,3.6165627983167337) node {$s_{j-1}$};
\draw[color=black] (4.365877757403756,3.4213005266253007) node {$s_\ell$};
\draw[color=black] (4.021297277948285,4.051947501518089) node {$b_\ell$};
\draw[color=black] (0.4917162276839946,2.8733459033331897) node {$b_{j-1}$};
\end{scriptsize}
\end{tikzpicture}
\end{center}
\caption{The graph $G$ after tails are removed. }
\label{sf:step2}
\end{subfigure}
\vspace{.5cm}

\begin{subfigure}[b]{\textwidth}
\begin{center}
\begin{tikzpicture}[line cap=round,line join=round,>=triangle 45,x=1.0cm,y=1.0cm,scale = 1.5]
\draw [line width=.8pt] (1.,3.)-- (1.5,3.866025403784439);
\draw [line width=.8pt] (1.5,3.866025403784439)-- (2.,3.);
\draw [line width=.8pt] (2.5,3.8660254037844384)-- (3.,3.);
\draw [line width=.8pt] (3.,3.)-- (3.5,3.866025403784437);
\draw [line width=.8pt] (4.,3.)-- (4.5,3.8660254037844353);
\draw [line width=.8pt] (4.5,3.8660254037844353)-- (3.5,3.866025403784437);
\draw [line width=.8pt] (3.5,3.866025403784437)-- (2.5,3.8660254037844384);
\draw [line width=.8pt,dotted] (2.5,3.8660254037844384)-- (1.5,3.866025403784439);
\draw [line width=.8pt] (1.5,3.866025403784439)-- (0.5,3.8660254037844393);
\draw [line width=.8pt] (0.5,3.8660254037844393)-- (1.,3.);
\draw [line width=.8pt] (1.,3.)-- (2.,3.);
\draw [line width=.8pt,dotted] (2.,3.)-- (3.,3.);
\draw [line width=.8pt] (3.,3.)-- (4.,3.);
\begin{scriptsize}
\draw [fill=black] (1.,3.) circle (1.5pt);
\draw[color=black] (0.9832460507492192,2.864228751505624) node {$j+1$};
\draw [fill=black] (0.5,3.8660254037844393) circle (1.5pt);
\draw[color=black] (0.46637533156601324,4.093232461563467) node {$j$};
\draw [fill=black] (1.5,3.866025403784439) circle (1.5pt);
\draw[color=black] (1.477144737968727,4.093232461563467) node {$j+2$};
\draw [fill=black] (2.,3.) circle (1.5pt);
\draw[color=black] (1.9940154571519328,2.864228751505624) node {$j+3$};
\draw [fill=black] (2.5,3.8660254037844384) circle (1.5pt);
\draw[color=black] (2.487914144371441,4.116204493527165) node {$j+k-2$};
\draw [fill=black] (3.,3.) circle (1.5pt);
\draw[color=black] (2.9818128315909487,2.841256719541926) node {$j+k-1$};
\draw [fill=black] (3.5,3.866025403784437) circle (1.5pt);
\draw[color=black] (3.4642255028286075,4.116204493527165) node {$j+k$};
\draw [fill=black] (4.,3.) circle (1.5pt);
\draw[color=black] (4.02704028593921,2.841256719541926) node {$j+k+1$};
\draw [fill=black] (4.5,3.8660254037844353) circle (1.5pt);
\draw[color=black] (4.509452957176869,4.090662541472712) node {$\#$};
\draw[color=black] (4.365877757403756,3.4213005266253007) node {$s_\ell$};
\draw[color=black] (4.021297277948285,4.051947501518089) node {$b_\ell$};
\draw[color=black] (0.2917162276839946,3.2733459033331897) node {$s_{j-1} + b_{j-1}$};
\end{scriptsize}
\end{tikzpicture}
\end{center}
\caption{The graph $G$ after a series transformation}
\label{sf:step3}
\end{subfigure}
\vspace{.5cm}

\begin{subfigure}[b]{\textwidth}
\begin{center}
\begin{tikzpicture}[line cap=round,line join=round,>=triangle 45,x=1.0cm,y=1.0cm,scale = 1.5]
\draw [line width=.8pt] (2.5,3.8660254037844384)-- (3.,3.);
\draw [line width=.8pt] (4.,3.)-- (4.5,3.8660254037844353);
\draw [line width=.8pt] (4.5,3.8660254037844353)-- (3.5,3.866025403784437);
\draw [line width=.8pt] (3.5,3.866025403784437)-- (2.5,3.8660254037844384);
\draw [line width=.8pt] (2.5,3.8660254037844384)-- (1.5,3.866025403784439);
\draw [line width=.8pt,dotted] (1.5,3.866025403784439)-- (0.5,3.8660254037844393);
\draw [line width=.8pt] (-.5,3.866025403784439)-- (0.5,3.8660254037844393);
\draw [line width=.8pt] (-1.5,3.866025403784439)-- (0.5,3.8660254037844393);
\draw [line width=.8pt] (3.,3.)-- (4.,3.);
\begin{scriptsize}
\draw [fill=black] (-1.5,3.8660254037844393) circle (1.5pt);
\draw[color=black] (-1.46637533156601324,4.093232461563467) node {$j$};
\draw [fill=black] (-0.5,3.8660254037844393) circle (1.5pt);
\draw[color=black] (-0.46637533156601324,4.093232461563467) node {$j+1$};
\draw [fill=black] (0.5,3.8660254037844393) circle (1.5pt);
\draw[color=black] (0.46637533156601324,4.093232461563467) node {$j+2$};
\draw [fill=black] (1.5,3.866025403784439) circle (1.5pt);
\draw[color=black] (1.477144737968727,4.093232461563467) node {$j+k-2$};
\draw [fill=black] (2.5,3.8660254037844384) circle (1.5pt);
\draw[color=black] (2.487914144371441,4.116204493527165) node {$\ast$};
\draw [fill=black] (3.,3.) circle (1.5pt);
\draw[color=black] (2.8818128315909487,2.841256719541926) node {$j+k-1\;$};
\draw [fill=black] (3.5,3.866025403784437) circle (1.5pt);
\draw[color=black] (3.4642255028286075,4.116204493527165) node {$j+k$};
\draw [fill=black] (4.,3.) circle (1.5pt);
\draw[color=black] (4.12704028593921,2.841256719541926) node {$j+k+1$};
\draw [fill=black] (4.5,3.8660254037844353) circle (1.5pt);
\draw[color=black] (4.509452957176869,4.080662541472712) node {$\#$};
\draw[color=black] (-1,3.72792638278866) node {$t_{1,j-1}$};
\draw[color=black] (0,3.72792638278866) node {$t_{2,j-1}$};
\draw[color=black] (2,3.72792638278866) node {$t_{k-1,j-1}$};

\draw[color=black] (2.665947392090101,3.3523844307342068) node {$s_{k-1,j-1}$};
\draw[color=black] (3.171332095291458,3.7314229581352234) node {$b_{k-1,j-1}$};
\draw[color=black] (3.5,3.125356462697905) node {$1$};
\draw[color=black] (4.465877757403756,3.4213005266253007) node {$s_\ell$};
\draw[color=black] (4.021297277948285,4.051947501518089) node {$b_\ell$};
\end{scriptsize}
\end{tikzpicture}
\end{center}
\caption{The graph $G$ after $k-1$ $\Delta$--Y transformations on the left. }

\label{sf:step4}
\end{subfigure}

%
%
\caption{The process used to find resistance distance for an arbitrary pair of nodes in the straight linear 2-tree $G$. Since $\Delta$--Y transformations are performed on the right and left hand sides of $G$, we use $\ast$ to denote the central node in the $Y$ on the left of the graph, and $\#$ for the central $Y$ node on the right.}
\label{fig:bothsides}
\end{figure}

	 \begin{proof}	 
	 First we apply $p=j-1\,$ $\Delta$--Y transformations, on the left-most triangles, and $\ell = m-j-k+1$ transformations on the right-most triangles, as seen in Figure~\ref{fig:bothsides}(a).    By the cut vertex theorem, we can ignore the tails on both sides of the graph, as in Figure~\ref{fig:bothsides}(b).  When we use a series transformation on the leftmost two remaining edges, we obtain the graph in Figure~\ref{fig:bothsides}(c), which by Lemma \ref{lem:splusb} has edge weight
	 \[
	 w(j,j+1) = \frac{F_{2j-1}}{F_{2j}}.
	 \]
Finally, we perform $k-1\,$ $\Delta$--Y transformations to obtain the graph in Figure~\ref{fig:bothsides}(d). Lemma \ref{lem3g}	gives the first equality in the theorem.  For the second equality, note that 
\[\begin{array}{c}
	\ds{
	\left(\frac{1}{b_{k-1,j-1}} +\frac{1}{1+s_{k-1,j-1}+s_{m-j-k+1}+b_{m-j-k+1} }\right)^{-1}
	}\\[6mm]
	\ds{
	=\left(\frac{{F_{2j+2k-2}}}{{F_{k}F_{2j+k-2}}} +\frac{1}{1+\frac{F_{k-1}F_{2j+k-3}}{F_{2j+2k-2}}+\frac{F_{2m-2j-2k+3}}{F_{2m-2j-2k+4}} }\right)^{-1}
	}\\[6mm]
	\ds{
		=\left(\frac{{F_{2j+2k-2}}}{{F_{k}F_{2j+k-2}}} +\frac{{F_{2j+2k-2}F_{2m-2j-2k+4}} }{{F_{2m-2j-2k+4}F_{2j+2k-2} + F_{2m-2j-2k+4}F_{k-1}F_{2j+k-3}+F_{2m-2j-2k+3}F_{2j+2k-2}}}\right)^{-1}
	}\\[6mm]
	\ds{
		=\left(\frac{{F_{2j+2k-2}}}{{F_{k}F_{2j+k-2}}} +\frac{{F_{2j+2k-2}F_{2m-2j-2k+4}} }{{ F_{2m-2j-2k+4}F_{k-1}F_{2j+k-3}+F_{2m-2j-2k+5}F_{2j+2k-2}}}\right)^{-1}
	}\\[6mm]
	\ds{
		=\frac{(F_{2m-2j-2k+4}F_{k-1}F_{2j+k-3}+F_{2m-2j-2k+5}F_{2j+2k-2})F_{k}F_{2j+k-2}}{F_{2j+2k-2}(F_{2m-2j-2k+4}F_{k-1}F_{2j+k-3}+F_{2m-2j-2k+5}F_{2j+2k-2}) + F_{2j+2k-2}F_{2m-2j-2k+4}F_{k}F_{2j+k-2}} 
	}\\[6mm]
	\ds{
		=\frac{(F_{2m-2j-2k+4}F_{k-1}F_{2j+k-3}+F_{2m-2j-2k+5}F_{2j+2k-2})F_{k}F_{2j+k-2}}
		{F_{2j+2k-2}(F_{2m-2j-2k+4}F_{k-1}F_{2j+k-3}+F_{2m-2j-2k+5}F_{2j+2k-2}+ F_{2m-2j-2k+4}F_{k}F_{2j+k-2}) } 
	}\\[6mm]
	\ds{
		=\frac{(F_{2m-2j-2k+4}F_{k-1}F_{2j+k-3}+F_{2m-2j-2k+5}F_{2j+2k-2})F_{k}F_{2j+k-2}}
		{F_{2j+2k-2}(F_{2m-2j-2k+4}F_{2j+2k-3}+F_{2m-2j-2k+5}F_{2j+2k-2}) } 
	}\\[6mm]
	\ds{
		\frac{(F_{2m-2j-2k+4}F_{k-1}F_{2j+k-3}+F_{2m-2j-2k+5}F_{2j+2k-2})F_{k}F_{2j+k-2}}
		{F_{2j+2k-2}F_{2m+2} } 
	}\\[6mm]
	  \end{array}
\]
	 
	 \end{proof}

We are now prepared to prove our main result.
\begin{theorem}\label{thm:main}
 Given a straight linear 2-tree with $n$ vertices and $m = n-2$ cells, the resistance distance between node $j<n$ and node $j+k \leq n$, $k \geq 1$, is
	 \begin{multline}\label{eq:wayneformulagen}r_m(j,j+k)=\frac{\sum_{i=1}^{k}\left[F_iF_{i+2j-2} - F_{i-1}F_{i+2j-3}\right]F_{2m-2i-2j+5}}{F_{2m+2}}\\
	= \frac{F_{m+1}^2+F_k^2F_{m-2j-k+3}^2+\frac{F_{m+1}}{5}\left[{F_{m-k}}(kL_k-F_k)+{F_{m-k+1}}\left((k-5)F_{k+1}+(2k+2)F_{k}\right)\right]}{F_{2m+2}}.\end{multline}

	 \end{theorem}
	 \begin{proof}
We first show that 
	 \begin{equation}
	 \label{eq:wayneformulagensum}r_m(j,j+k)=\frac{\sum_{i=1}^{k}\left[F_iF_{i+2j-2} - F_{i-1}F_{i+2j-3}\right]F_{2m-2i-2j+5}}{F_{2m+2}}.
	 \end{equation}

Fixing $n$ and $1\leq j < n$, we induct on $k$.  
	 
	\begin{figure}[h!]
\begin{center}
\begin{tikzpicture}[line cap=round,line join=round,>=triangle 45,x=1.0cm,y=1.0cm]
\draw [line width=.8pt] (-3.,2.)-- (-2.,2.);
\draw [line width=.8pt] (-2.,2.)-- (-1.,2.);
\draw [line width=.8pt,dotted] (-1.,2.)-- (0.,2.);
\draw [line width=.8pt] (0.,2.)-- (1.,2.);
\draw [line width=.8pt] (1.,2.)-- (2.,3.);
\draw [line width=.8pt] (1.,2.)-- (2.,1.);
\draw [line width=.8pt] (2.,1.)-- (3.5,1.);
\draw [line width=.8pt] (2.,3.)-- (3.5,3.);
\draw [line width=.8pt] (3.5,3.)-- (4.5,3.);
\draw [line width=.8pt,dotted] (4.5,3.)-- (5.5,3.);
\draw [line width=.8pt] (5.5,3.)-- (6.52,3.);
\draw [line width=.8pt] (3.5,1.)-- (3.5,3.);
\draw [line width=.8pt] (5.5,3.)-- (6.52,3.);
\begin{scriptsize}
\draw [fill=black] (-3.,2.) circle (1.5pt);
\draw[color=black] (-3.02,2.33) node {$1$};
\draw [fill=black] (-2.,2.) circle (1.5pt);
\draw[color=black] (-2.04,2.37) node {$2$};
\draw [fill=black] (-1.,2.) circle (1.5pt);
\draw[color=black] (-1.06,2.39) node {$3$};
\draw [fill=black] (0.,2.) circle (1.5pt);
\draw[color=black] (-0.06,2.37) node {$j-1$};
\draw [fill=black] (1.,2.) circle (1.5pt);
\draw[color=black] (0.92,2.35) node {$\ast$};
\draw [fill=black] (2.,3.) circle (1.5pt);
\draw[color=black] (1.96,3.35) node {$j+1$};
\draw [fill=black] (2.,1.) circle (1.5pt);
\draw[color=black] (1.8,0.85) node {$j$};
\draw [fill=black] (3.5,1.) circle (1.5pt);
\draw[color=black] (3.88,0.85) node {$j+2$};
\draw [fill=black] (3.5,3.) circle (1.5pt);
\draw[color=black] (3.52,3.36) node {$\#$};
\draw [fill=black] (4.5,3.) circle (1.5pt);
\draw[color=black] (4.5,3.41) node {$j+3$};
\draw [fill=black] (5.5,3.) circle (1.5pt);
\draw[color=black] (5.44,3.37) node {$n-1$};
\draw [fill=black] (6.52,3.) circle (1.5pt);
\draw[color=black] (6.48,3.33) node {$n$};
\draw[color=black] (-2.44,1.85) node {$t_1$};
\draw[color=black] (-1.44,1.85) node {$t_2$};
\draw[color=black] (0.56,1.85) node {$t_{j-1}$};
\draw[color=black] (1.84,2.43) node {$b_{j-1}$};
\draw[color=black] (1.28,1.37) node {$s_{j-1}$};
\draw[color=black] (2.8,0.85) node {$1$};
\draw[color=black] (2.72,2.87) node {$b_{m-j}$};
\draw[color=black] (4.06,2.85) node {$t_{m-j}$};
\draw[color=black] (6.06,2.85) node {$t_1$};
\draw[color=black] (3.92,2.07) node {$s_{m-j}$};
\end{scriptsize}
\end{tikzpicture}
\end{center}
\caption{A linear 2-tree after $p=j-1\;$ $\Delta$--Y transformations are done on the left and $n-j-2=m-j$ transformations are done on the right.}
\label{fig:jplus1}
\end{figure}
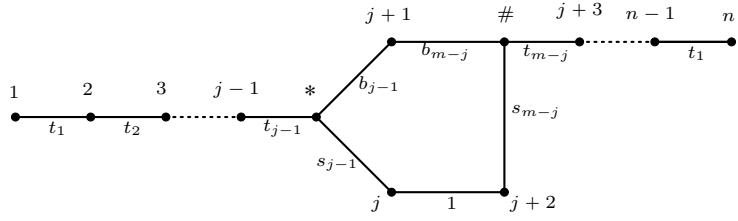 
	 
	 Base Case:  Let $k=1$.  Applying $\Delta$-Y transformations to both sides of the linear 2-tree yields the situation shown in Figure~\ref{fig:jplus1}.  The resistance distance $r_m(j,j+1)$ is given by a simple application of the parallel rule coupled with the cut vertex theorem.
	 \[\begin{array}{r}
\ds{
	 r_m(j,j+1)=\left(\frac{1}{{b_{j-1}}+{s_{j-1}}}+\frac{1}{{b_{m-j}}+{s_{m-j}}+1}\right)^{-1}
	 =\left(\frac{F_{2j}}{F_{2j-1}}+\frac{1}{\frac{F_{2m-2j+1}}{F_{2m-2j+2}}+1}\right)^{-1}
}
	 \\[6mm]
\ds{
	 =\left(\frac{F_{2j}}{F_{2j-1}}+\frac{F_{2m-2j+2}}{F_{2m-2j+3}}\right)^{-1}
	 =\frac{F_{2j-1}F_{2m-2j+3}}{F_{2j}F_{2m-2j+3}+F_{2m-2j+2}F_{2j-1}}
	 =\frac{F_{2j-1}F_{2m-2j+3}}{F_{2m+2}}
}
	\end{array}\]
	 as desired.

	We now assume that~\eqref{eq:wayneformulagensum} holds for $k$ and show it also holds for $k+1$.  By Lemma~\ref{lem43}
	
	\[\begin{aligned}r_m(j,j+k+1)-r_m(j,j+k) &=\dfrac{F_{k}F_{k+1}F_{k+2j-2}F_{k+2j-1}}{F_{2k+2j-2}F_{2k+2j}} \\
	&\qquad + \frac{(F_{2m-2j-2k+2}F_{k}F_{2j+k-2}+F_{2m-2j-2k+3}F_{2j+2k})F_{k+1}F_{2j+k-1}}{F_{2j+2k}F_{2m+2}}
		 \\
		&  \qquad \qquad - \frac{(F_{2m-2j-2k+4}F_{k-1}F_{2j+k-3}+F_{2m-2j-2k+5}F_{2j+2k-2})F_{k}F_{2j+k-2}}{F_{2j+2k-2}F_{2m+2}}.\\
	\end{aligned}\]
	After getting common denominators, we have 
	
	\begin{multline*}
	r_m(j,j+k+1)-r_m(j,j+k) =
	\frac{F_{2m+2}F_{k}F_{k+1}F_{k+2j-2}F_{k+2j-1} }{F_{2j+2k}F_{2j+2k-2}F_{2m+2}}\\
	+\frac{F_{2j+2k-2} (F_{2m-2j-2k+2}F_{k}F_{2j+k-2}+F_{2m-2j-2k+3}F_{2j+2k})F_{k+1}F_{2j+k-1}}{F_{2j+2k}F_{2j+2k-2}F_{2m+2}}\\
	-\frac{F_{2j+2k}(F_{2m-2j-2k+4}F_{k-1}F_{2j+k-3}+F_{2m-2j-2k+5}F_{2j+2k-2})F_{k}F_{2j+k-2}}{F_{2j+2k}F_{2j+2k-2}F_{2m+2}}
	\end{multline*}	
	
	\begin{multline}
	=\frac{F_{2m+2}F_{k}F_{k+1}F_{k+2j-2}F_{k+2j-1} +F_{2j+2k-2} F_{2m-2j-2k+2}F_{k}F_{2j+k-2}F_{k+1}F_{2j+k-1}}{F_{2j+2k}F_{2j+2k-2}F_{2m+2}}\\
	-\frac{F_{2j+2k}F_{2m-2j-2k+4}F_{k-1}F_{2j+k-3}F_{k}F_{2j+k-2}}{F_{2j+2k}F_{2j+2k-2}F_{2m+2}}\\
	+\frac{F_{2m-2j-2k+3}F_{k+1}F_{2j+k-1}-F_{2m-2j-2k+5}F_{k}F_{2j+k-2}}{F_{2m+2}}\\
	\end{multline}	
	
Given that $F_{2m+2} = F_{2m-2j-2k+4}F_{2j+2k} - F_{2m-2j-2k+2}F_{2j+2k-2}$ by Proposition~\ref{prop:splitdiff}, we combine the numerators of the first and second fractions to obtain 
\begin{multline}
F_{2j+2k-2}(F_{2m-2j-2k+2}F_{k}F_{2j+k-2}F_{k+1}F_{2j+k-1} - F_{2m-2j-2k+2}F_{k}F_{k+1}F_{k+2j-2}F_{k+2j-1} ) \\
+F_{2j+2k}(F_{2m-2j-2k+4}F_{k}F_{k+1}F_{k+2j-2}F_{k+2j-1}-F_{2m-2j-2k+4}F_{k-1}F_{2j+k-3}F_{k}F_{2j+k-2})\\
=F_{2j+2k}F_{2m-2j-2k+4}F_{k}F_{k+2j-2}(F_{k+1}F_{k+2j-1}-F_{k-1}F_{2j+k-3})\\
=F_{2j+2k}F_{2m-2j-2k+4}F_{k}F_{k+2j-2}(F_{2k+2j-2})\\
\end{multline}

Thus, 
\begin{multline*}
r_m(j,j+k+1)-r_m(j,j+k)
=	\frac{F_{2m-2j-2k+4}F_{k}F_{k+2j-2}+F_{2m-2j-2k+3}F_{k+1}F_{2j+k-1}-F_{2m-2j-2k+5}F_{k}F_{2j+k-2}}{F_{2m+2}}\\
=	\frac{\left[F_{k+1}F_{k+2j-1} - F_{k}F_{k+2j-2}\right]F_{2m-2k-2j+3}}{F_{2m+2}}\\
\end{multline*}

Adding this result to 
\[
\frac{\sum_{i=1}^{k}\left[F_iF_{i+2j-2} - F_{i-1}F_{i+2j-3}\right]F_{2m-2i-2j+5}}{F_{2m+2}}
\]
yields
\[
\frac{\sum_{i=1}^{k+1}\left[F_iF_{i+2j-2} - F_{i-1}F_{i+2j-3}\right]F_{2m-2i+5-2j}}{F_{2m+2}}
\] as desired.  


We now prove the second equality of the theorem by noting that if 
\[
C_k = \frac{F_{m+1}^2+F_k^2F_{m-2j-k+3}^2+F_{m+1}\left({F_{m-k}}(kL_k-F_k)+{F_{m-k+1}}\left((k-5)F_{k+1}+(2k+2)F_{k}\right)\right)/5}{F_{2m+2}}
\]	
then $\displaystyle C_1 = \frac{F_{m+1}^2+F_{m-2j+2}^2}{F_{2m+2}} = \frac{F_{2j-1}F_{2m-2j+3}}{F_{2m+2}}$ by Catalan's Identity and
\[\begin{aligned}
C_{k+1} - C_{k}
&=\frac{F_{k+1}^2F_{m-2j-k+2}^2-F_k^2F_{m-2j-k+3}^2 }{F_{2m+2}}+\frac{F_{m+1}\left(F_{m-k-1}((k+1)L_{k+1}-F_{k+1})\right)}{5F_{2m+2}}\\
&\qquad +\frac{F_{m+1}}{5F_{2m+2}}\left(
F_{m-k}\left((k-4)F_{k+2}+(2k+4)F_{k+1}-kL_k+F_k\right)
-F_{m-k+1}\left((k-5)F_{k+1}+(2k+2)F_{k}\right)\right).\end{aligned}\]
Replacing $F_{m-k+1}$ by $F_{m-k} + F_{m-k-1}$, $F_{k+2}$ by $F_{k+1} + F_k$ and $L_{k+1}$ by $F_{k+2}+F_k$, and $L_k$ by $F_{k+1}+F_{k-1}$ gives \[\begin{aligned}
C_{k+1}=C_k
&=\frac{F_{k+1}^2F_{m-2j-k+2}^2-F_k^2F_{m-2j-k+3}^2 }{F_{2m+2}}\\
&\qquad+\frac{
F_{m+1}}{5F_{2m+2}}
\left(F_{m-k-1}((k+1)(F_{k+2}+F_k)-F_{k+1}-(k-5)F_{k+1}-(2k+2)F_{k})\right)\\
&\qquad\qquad+\frac{
F_{m+1}}{5F_{2m+2}}
\left(F_{m-k}\left(2kF_{k+1}+(k-3)F_k-kF_{k-1}-(k-5)F_{k+1}-(2k+2)F_{k}\right)\right)\\[2mm]
&=\frac{F_{k+1}^2F_{m-2j-k+2}^2-F_k^2F_{m-2j-k+3}^2 +F_{m+1}F_{m-k-1}F_{k+1}+F_{m+1}F_{m-k}F_{k-1}}{F_{2m+2}}\\[2mm]
&=\frac{F_{k+1}^2F_{m-2j-k+2}^2-F_k^2F_{m-2j-k+3}^2 +F_{m+1}F_{m-k+1}F_{k+1}-F_{m+1}F_{m-k}F_{k}}{F_{2m+2}}\\
\end{aligned}\]
Notice that the numerator above takes the form $F_{k+1}A - F_{k}B$, where

\[
A = F_{k+1}F_{m-2j-k+2}^2 + F_{m+1}F_{m-k+1}, \text{ and } B = F_kF_{m-2j-k+3}^2 + F_{m+1}F_{m-k}. 
\]
Proposition~\ref{app:simpA} and Proposition~\ref{app:simpB} demonstrate that 
\[
A = F_{2m-2j-2k+3}F_{2j+k-1}+F_kF_{m-2j-k+2}F_{m-2j-k+3}\]
and
\[B = F_{2j+k-2}  F_{2m-2j-2k+3}+ F_{k+1}F_{m-2j-k+3} F_{m-2j-k+2}.\]
Thus, 
\[\begin{aligned}
C_{k+1}-C_{k} &= (F_{k+1}A - F_kB)/F_{2m+2} \\
&= F_{2m-2j-2k+3}[F_{k+1}F_{2j+k-1}-F_k F_{2j+k-2} ]/F_{2m+2}\\
&=r_m(j,j+k+1) - r_m(j,j+k)
\end{aligned}\]
which gives the second equality. 

	\end{proof}

\section{Spanning trees and 2-forests in straight linear 2-trees}\label{sec:spanningtrees}

\begin{theorem}\label{thm:spanningtrees}
The number of spanning trees in the straight linear 2-tree with $m$ triangles is $F_{2m+2}$.
\end{theorem}
\begin{proof}
For ease of notation we will let $\mathscr{S}_m = $ the collection of spanning trees in a  linear 2-tree  with $m$ triangles. 
It is straightforward to verify that $|\mathscr{S}_1| = 3 = F_4$.  Now suppose that $|\mathscr{S}_k| = F_{2k+2}$ for all $k<m$. 
Consider a linear 2-tree $G_n$ with $n$ vertices and $m=n-2$ triangles, labeling the exterior edges from left to right, as shown in the figure below.  

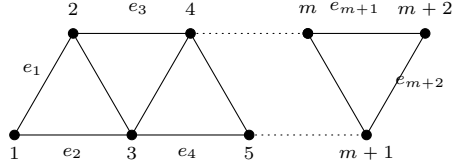
\begin{center}
\begin{tikzpicture}[line cap=round,line join=round,>=triangle 45,x=1.0cm,y=1.0cm,scale = 1.2]
\draw (-2.,1.)-- (-0.7,1.);
\draw (-0.7,1.)-- (-1.35,2.1258330249197703);
\draw (-1.35,2.1258330249197703)-- (-2.,1.);
\draw (-0.05,2.1258330249197703)-- (-1.35,2.1258330249197703);
\draw (-0.05,2.1258330249197703)-- (-0.7,1.);
\draw (-0.7,1.)-- (0.6,1.);
\draw (-0.05,2.1258330249197703)-- (0.6,1.);
\draw[dotted] (-0.05,2.1258330249197703)-- (1.25,2.12583302491977);
\draw[dotted] (0.6,1.)-- (1.9,1.);


\draw (1.9,1.)-- (1.25,2.12583302491977);
\draw (1.9,1.)-- (2.55,2.1258330249197694);
\draw (2.55,2.1258330249197694)-- (1.25,2.1258330249197703);
\begin{scriptsize}
\draw [fill=black] (-2.,1.) circle (1.5pt);
\draw [color = black] (-2.,.8) node {$1$};
\draw [fill=black] (-0.7,1.) circle (1.5pt);
\draw [color = black] (-.7,.8) node {$3$};
\draw[color=black] (-1.3558677685950407,0.8) node {$e_2$};
\draw [fill=black] (-1.35,2.1258330249197703) circle (1.5pt);
\draw [color = black] (-1.35,2.4) node {$2$};
\draw[color=black] (-1.8,1.725619834710745) node {$e_1$};
\draw [fill=black] (-0.05,2.1258330249197703) circle (1.5pt);
\draw [color = black] (-.05,2.4) node {$4$};
\draw[color=black] (-0.6285950413223138,2.4) node {$e_3$};
\draw [fill=black] (0.6,1.) circle (1.5pt);
\draw [color = black] (0.6,.8) node {$5$};
\draw[color=black] (-0.08314049586776859,0.8) node {$e_4$};
\draw [fill=black] (1.25,2.12583302491977) circle (1.5pt);
\draw [color = black] (1.25,2.4) node {$m$};
\draw [fill=black] (1.9,1.) circle (1.5pt);
\draw [color = black] (1.9,.8) node {$m+1$};
\draw [fill=black] (2.55,2.1258330249197694) circle (1.5pt);
\draw [color = black] (2.55,2.4) node {$m+2$};
\draw[color=black] (2.495371900826445,1.5768595041322329) node {$e_{m+2}$};
\draw [fill=black] (1.25,2.1258330249197703) circle (1.5pt);
\draw[color=black] (1.7680991735537184,2.4) node {$e_{m+1}$};
\end{scriptsize}
\end{tikzpicture}
\end{center}

Notice that no spanning tree can contain all of the exterior edges. 
Thus, we can partition $\mathscr{S}_m$ into $m+2$ sets as follows.  
\begin{center}
\begin{tabular}{ll}
$\mathscr{C}_1$ = trees in $\mathscr{S}_m$ which contain $e_2, e_3, \ldots, e_{m+2}$ but do not contain $e_1$. \\[2mm]
$\mathscr{C}_2$ = trees in $\mathscr{S}_m$ which contain $e_3, e_4, \ldots, e_{m+2}$ but do not contain $e_2$. \\[2mm]
$\mathscr{C}_3$ = trees in $\mathscr{S}_m$ which contain $e_4, e_5, \ldots, e_{m+2}$ but do not contain $e_3$.\\[2mm]
$\mathscr{C}_4$ = trees in $\mathscr{S}_m$ which contain $e_5, e_6, \ldots, e_{m+2}$ but do not contain $e_4$.\\ [2mm]
\ \ \ $\vdots$ \\[2mm]
$\mathscr{C}_{m+2}$ = trees in $\mathscr{S}_m$ which do not contain $e_{m+2}$.  \\ [2mm]
\end{tabular}
\end{center}

In general, if $k \leq m$ and the edges $e_{k+1}, \ldots, e_{m+2}$ are each contained in a spanning tree $T$ for $G_m$, then none of the diagonal edges $\{i,i+1\}$ for $i \geq k$ can be contained in $T$.  Thus, the trees contained in $\mathscr{C}_k$ for $2<k \leq m$ are exactly the same as the spanning trees of the graph $G_m$ with the edges $e_k$ and $\{i,i+1\}_{i\geq k}$ removed.   By examining the Table~\ref{tab:tree}, it is apparent that the number of spanning trees in $\mathscr{C}_k$ for $k \geq 2$ is the same as the number of spanning trees in $G_{k-2}$, which, according to our inductive hypothesis, is exactly $F_{2(k-2) + 2}$. 
Finally, we conclude that 
\[
|\mathscr{S}_m| = 1 + F_2 + F_4 + \ldots + F_{2m-2} + 2F_{2m} = F_{2m+2}.
\]
\end{proof}
Theorems \ref{thm:main} and  \ref{thm:spanningtrees} together with Lemma 2 in~\cite{BapatWheels} give the following combinatorial result. 

\begin{theorem}
Let $G$ be a straight linear 2-tree with $n = m+2$ vertices.  Then the number of spanning 2-forests of $G$ which separate nodes $j$ and $j+k$ is 

\begin{multline}\label{eq:wayneformulagen2}
	\sum_{i=1}^{k}\left[F_iF_{i+2j-2} - F_{i-1}F_{i+2j-3}\right]F_{2m-2i-2j+5}\\
	= F_{m+1}^2+F_k^2F_{m-2j-k+3}^2+\frac{F_{m+1}}{5}\left({F_{m-k}}(kL_k-F_k)
	+F_{m-k+1}\left((k-5)F_{k+1}+(2k+2)F_{k}\right)\right).
\end{multline}

\end{theorem}

\begin{table}[!ht]
\begin{center}
\begin{tabular}{| c | c | c |}
\hline 
$k$ & ${G}_m$ with relevant edges removed & $|\mathscr{C}_k|$ \\\hline
1 & \begin{tikzpicture}[line cap=round,line join=round,>=triangle 45,x=1.0cm,y=1.0cm]
\draw (-1.35,2.1258330249197703)-- (-0.05,2.1258330249197703);
\draw[dotted] (1.25,2.12583302491977)-- (-0.05,2.1258330249197703);
\draw (1.25,2.12583302491977)-- (2.55,2.1258330249197694);
\draw (2.55,2.1258330249197694)-- (1.9,1.);
\draw (-2.,1.)-- (-0.7,1.);
\draw[dotted] (-0.7,1.)-- (0.6,1.);

\draw (0.6,1.)-- (1.9,1.);
\begin{scriptsize}
\draw [fill=black] (-2.,1.) circle (1.5pt);
\draw [fill=black] (-0.7,1.) circle (1.5pt);
\draw [fill=black] (-1.35,2.1258330249197703) circle (1.5pt);
\draw [fill=black] (-0.05,2.1258330249197703) circle (1.5pt);
\draw [fill=black] (0.6,1.) circle (1.5pt);
\draw [fill=black] (1.25,2.12583302491977) circle (1.5pt);
\draw [fill=black] (1.9,1.) circle (1.5pt);
\draw [fill=black] (2.55,2.1258330249197694) circle (1.5pt);
\draw [fill=black] (1.25,2.1258330249197703) circle (1.5pt);
\draw[color=black] (-0.711239669421486,2.4) node {$e_3$};

\draw[color=black] (1.8672727272727279,2.4) node {$e_{m+1}$};
\draw[color=black] (2.613553719008265,1.4942148760330594) node {$e_{m+2}$};
\draw[color=black] (-1.2732231404958658,0.8) node {$e_2$};
\draw[color=black] (1.3218181818181827,0.8) node {$e_m$};
\end{scriptsize}
\end{tikzpicture}
 & 1 \\   \hline
 2 & 
 \begin{tikzpicture}[line cap=round,line join=round,>=triangle 45,x=1.0cm,y=1.0cm]
 \draw (2.55,2.1258330249197694)-- (3.85,2.125833024919768);
\draw (-1.35,2.1258330249197703)-- (-0.05,2.1258330249197703);
\draw (1.25,2.12583302491977)-- (-0.05,2.1258330249197703);
\draw[dotted] (1.25,2.12583302491977)-- (2.55,2.1258330249197694);

\draw (-0.7,1.)-- (0.6,1.);
\draw (3.2,1.)-- (3.85,2.125833024919768);
\draw[dotted] (0.6,1.)-- (1.9,1.);
\draw (1.9,1.)-- (3.2,1.);

\draw (-1.35,2.1258330249197703)-- (-2.,1.);
\begin{scriptsize}
\draw [fill=black] (-2.,1.) circle (1.5pt);
\draw [fill=black] (-0.7,1.) circle (1.5pt);
\draw [fill=black] (-1.35,2.1258330249197703) circle (1.5pt);
\draw [fill=black] (-0.05,2.1258330249197703) circle (1.5pt);
\draw[color=black] (0.4623140495867786,2.4033057851239676) node {$e_5$};
\draw [fill=black] (0.6,1.) circle (1.5pt);
\draw[color=black] (-0.19884297520660965,0.8) node {$e_4$};
\draw [fill=black] (1.25,2.12583302491977) circle (1.5pt);
\draw [fill=black] (1.9,1.) circle (1.5pt);
\draw [fill=black] (2.55,2.1258330249197694) circle (1.5pt);
\draw [fill=black] (3.2,1.) circle (1.5pt);
\draw [fill=black] (3.85,2.125833024919768) circle (1.5pt);
\draw[color=black] (3.14,2.4) node {$e_{m+1}$};
\draw[color=black] (-0.7608264462809896,2.4) node {$e_3$};
\draw[color=black] (3.9342148760330583,1.626446280991737) node {$e_{m+2}$};
\draw[color=black] (2.5614876033057863,0.8) node {$e_m$};
\draw[color=black] (-1.8352066115702454,1.890909090909092) node {$e_1$};
\end{scriptsize}
\end{tikzpicture} & $F_2$ \\ \hline
3 & 
\begin{tikzpicture}[line cap=round,line join=round,>=triangle 45,x=1.0cm,y=1.0cm]
\draw (-2.,1.)-- (-0.7,1.);
\draw (-1.35,2.1258330249197703)-- (-0.7,1.);
\draw (2.55,2.1258330249197694)-- (3.85,2.125833024919768);
\draw (1.25,2.12583302491977)-- (-0.05,2.1258330249197703);
\draw[dotted] (1.25,2.12583302491977)-- (2.55,2.1258330249197694);

\draw (-2.,1.)-- (-0.7,1.);
\draw (-0.7,1.)-- (0.6,1.);
\draw (3.2,1.)-- (3.85,2.125833024919768);
\draw[dotted] (0.6,1.)-- (1.9,1.);

\draw (1.9,1.)-- (3.2,1.);
\draw (-1.35,2.1258330249197703)-- (-2.,1.);
\begin{scriptsize}
\draw [fill=black] (-2.,1.) circle (1.5pt);
\draw [fill=black] (-0.7,1.) circle (1.5pt);
\draw[color=black] (-1.3558677685950389,0.8) node {$e_2$};
\draw [fill=black] (-1.35,2.1258330249197703) circle (1.5pt);
\draw [fill=black] (-0.05,2.1258330249197703) circle (1.5pt);
\draw[color=black] (0.4623140495867786,2.4) node {$e_5$};
\draw [fill=black] (0.6,1.) circle (1.5pt);
\draw [fill=black] (1.25,2.12583302491977) circle (1.5pt);
\draw [fill=black] (1.9,1.) circle (1.5pt);
\draw [fill=black] (2.55,2.1258330249197694) circle (1.5pt);
\draw [fill=black] (3.2,1.) circle (1.5pt);
\draw [fill=black] (3.85,2.125833024919768) circle (1.5pt);
\draw[color=black] (3.14,2.4) node {$e_{m+1}$};
\draw[color=black] (-1.8847933884297494,1.8082644628099185) node {$e_1$};
\draw[color=black] (0.032561983471076236,0.8) node {$e_4$};
\draw[color=black] (3.9342148760330583,1.626446280991737) node {$e_{m+2}$};
\draw[color=black] (2.5614876033057863,0.8991735537190106) node {$e_{m}$};
\end{scriptsize}
\end{tikzpicture}

& $F_{4}$ \\ \hline 


4 & 
\begin{tikzpicture}[line cap=round,line join=round,>=triangle 45,x=1.0cm,y=1.0cm]
\draw (-2.,1.)-- (-0.7,1.);
\draw (-1.35,2.1258330249197703)-- (-0.7,1.);
\draw (-0.05,2.1258330249197703)-- (-0.7,1.);
\draw (2.55,2.1258330249197694)-- (3.85,2.125833024919768);
\draw (-1.35,2.1258330249197703)-- (-0.05,2.1258330249197703);
\draw (1.25,2.12583302491977)-- (-0.05,2.1258330249197703);
\draw[dotted] (1.25,2.12583302491977)-- (2.55,2.1258330249197694);
\draw (-2.,1.)-- (-0.7,1.);
\draw (3.2,1.)-- (3.85,2.125833024919768);
\draw (0.6,1.)-- (1.9,1.);
\draw[dotted] (1.9,1.)-- (3.2,1.);
\draw (-1.35,2.1258330249197703)-- (-2.,1.);
\begin{scriptsize}
\draw [fill=black] (-2.,1.) circle (1.5pt);
\draw [fill=black] (-0.7,1.) circle (1.5pt);
\draw[color=black] (-1.3558677685950389,0.8) node {$e_2$};
\draw [fill=black] (-1.35,2.1258330249197703) circle (1.5pt);
\draw [fill=black] (-0.05,2.1258330249197703) circle (1.5pt);
\draw[color=black] (0.4623140495867786,2.4) node {$e_5$};
\draw[color=black] (-0.6451239669421467,2.3702479338842983) node {$e_3$};
\draw [fill=black] (0.6,1.) circle (1.5pt);
\draw [fill=black] (1.25,2.12583302491977) circle (1.5pt);
\draw [fill=black] (1.9,1.) circle (1.5pt);
\draw [fill=black] (2.55,2.1258330249197694) circle (1.5pt);
\draw [fill=black] (3.2,1.) circle (1.5pt);
\draw [fill=black] (3.85,2.125833024919768) circle (1.5pt);
\draw[color=black] (3.14,2.4) node {$e_{m+1}$};
\draw[color=black] (-1.9509090909090883,1.8247933884297534) node {$e_1$};
\draw[color=black] (3.9342148760330583,1.626446280991737) node {$e_{m+2}$};
\draw[color=black] (1.3218181818181831,0.8) node {$e_{6}$};
\end{scriptsize}
\end{tikzpicture}

& 

$F_6$ \\ \hline 

 & &  \\   
\vdots & \vdots & \vdots  \\ 
 & &   \\ \hline  

$m-1$ & \begin{tikzpicture}[line cap=round,line join=round,>=triangle 45,x=1.0cm,y=1.0cm]
\draw (-2.,1.)-- (-0.7,1.);
\draw (-0.7,1.)-- (-1.35,2.1258330249197703);
\draw (-1.35,2.1258330249197703)-- (-2.,1.);
\draw (-1.35,2.1258330249197703)-- (-0.7,1.);
\draw (-0.05,2.1258330249197703)-- (-0.7,1.);
\draw (1.25,2.12583302491977)-- (1.9,1.);
\draw (2.55,2.1258330249197694)-- (3.85,2.125833024919768);
\draw (-1.35,2.1258330249197703)-- (-0.05,2.1258330249197703);
\draw[dotted] (1.25,2.12583302491977)-- (-0.05,2.1258330249197703);
\draw (-2.,1.)-- (-0.7,1.);
\draw[dotted] (-0.7,1.)-- (0.6,1.);

\draw (3.2,1.)-- (3.85,2.125833024919768);
\draw (0.6,1.)-- (1.9,1.);
\draw (1.9,1.)-- (3.2,1.);
\draw (0.6,1.)-- (1.25,2.12583302491977);
\begin{scriptsize}
\draw [fill=black] (-2.,1.) circle (1.5pt);
\draw [fill=black] (-0.7,1.) circle (1.5pt);
\draw[color=black] (-1.3558677685950389,0.8) node {$e_2$};
\draw [fill=black] (-1.35,2.1258330249197703) circle (1.5pt);
\draw[color=black] (-1.836033057851237,1.725619834710745) node {$e_1$};
\draw [fill=black] (-0.05,2.1258330249197703) circle (1.5pt);
\draw [fill=black] (0.6,1.) circle (1.5pt);
\draw [fill=black] (1.25,2.12583302491977) circle (1.5pt);
\draw [fill=black] (1.9,1.) circle (1.5pt);
\draw [fill=black] (2.55,2.1258330249197694) circle (1.5pt);
\draw [fill=black] (3.2,1.) circle (1.5pt);
\draw [fill=black] (3.85,2.125833024919768) circle (1.5pt);
\draw[color=black] (3.14,2.4) node {$e_{m+1}$};
\draw[color=black] (-0.7608264462809896,2.4) node {$e_2$};
\draw[color=black] (4.0342148760330583,1.626446280991737) node {$e_{m+2}$};
\draw[color=black] (1.3218181818181831,0.8) node {$e_{m-2}$};
\draw[color=black] (2.5614876033057863,0.8) node {$e_m$};
\end{scriptsize}
\end{tikzpicture}

& $F_{2(m-3) + 2}$ \\ \hline 

$m$ & 

\begin{tikzpicture}[line cap=round,line join=round,>=triangle 45,x=1.0cm,y=1.0cm]
\draw(-2.,1.) -- (-0.7,1.) -- (-1.35,2.1258330249197703) -- cycle;
\draw(-0.05,2.1258330249197703) -- (-0.7,1.) -- (0.6,1.) -- cycle;
\draw (-2.,1.)-- (-0.7,1.);
\draw (-0.7,1.)-- (-1.35,2.1258330249197703);
\draw (-1.35,2.1258330249197703)-- (-2.,1.);
\draw (-0.05,2.1258330249197703)-- (-0.7,1.);
\draw (-0.7,1.)-- (0.6,1.);
\draw (0.6,1.)-- (-0.05,2.1258330249197703);
\draw (1.25,2.12583302491977)-- (1.9,1.);
\draw (2.55,2.1258330249197694)-- (3.85,2.125833024919768);
\draw (-1.35,2.1258330249197703)-- (-0.05,2.1258330249197703);
\draw[dotted] (-0.05,2.1258330249197703)-- (1.25,2.1258330249197703);

\draw[dotted] (0.6,1.)-- (1.9,1.);

\draw (1.25,2.12583302491977)-- (2.55,2.1258330249197694);
\draw (2.55,2.1258330249197694)-- (1.9,1.);
\draw (-2.,1.)-- (-0.7,1.);
\draw (-0.7,1.)-- (0.6,1.);
\draw (3.2,1.)-- (3.85,2.125833024919768);
\begin{scriptsize}
\draw [fill=black] (-2.,1.) circle (1.5pt);
\draw [fill=black] (-0.7,1.) circle (1.5pt);
\draw[color=black] (-1.3558677685950389,0.8) node {$e_2$};
\draw [fill=black] (-1.35,2.1258330249197703) circle (1.5pt);
\draw [fill=black] (-0.05,2.1258330249197703) circle (1.5pt);
\draw [fill=black] (0.6,1.) circle (1.5pt);
\draw [fill=black] (1.25,2.12583302491977) circle (1.5pt);
\draw [fill=black] (1.9,1.) circle (1.5pt);
\draw [fill=black] (2.55,2.1258330249197694) circle (1.5pt);
\draw [fill=black] (3.2,1.) circle (1.5pt);
\draw [fill=black] (3.85,2.125833024919768) circle (1.5pt);
\draw[color=black] (3.14,2.4) node {$e_{m+1}$};
\draw[color=black] (-0.7608264462809896,2.4) node {$e_3$};
\draw[color=black] (1.8507438016528936,2.4) node {$e_{m-1}$};
\draw[color=black] (-1.8847933884297494,1.8082644628099185) node {$e_1$};
\draw[color=black] (0.032561983471076236,0.8) node {$e_4$};
\draw[color=black] (3.99342148760330583,1.626446280991737) node {$e_{m+2}$};
\end{scriptsize}
\end{tikzpicture}

& $F_{2(m-2) + 2}$ \\ \hline 

$m + 1$ & 

\begin{tikzpicture}[line cap=round,line join=round,>=triangle 45,x=1.0cm,y=1.0cm]
\draw(-2.,1.) -- (-0.7,1.) -- (-1.35,2.1258330249197703) -- cycle;
\draw(-0.05,2.1258330249197703) -- (-0.7,1.) -- (0.6,1.) -- cycle;
\draw (-2.,1.)-- (-0.7,1.);
\draw (-0.7,1.)-- (-1.35,2.1258330249197703);
\draw (-1.35,2.1258330249197703)-- (-2.,1.);
\draw (-0.05,2.1258330249197703)-- (-0.7,1.);
\draw (-0.7,1.)-- (0.6,1.);
\draw (0.6,1.)-- (-0.05,2.1258330249197703);
\draw (-1.35,2.1258330249197703)-- (-0.05,2.1258330249197703);
\draw (1.25,2.12583302491977)-- (-0.05,2.1258330249197703);
\draw (2.55,2.1258330249197694)-- (1.9,1.);
\draw (-2.,1.)-- (-0.7,1.);
\draw (-0.7,1.)-- (0.6,1.);
\draw (3.2,1.)-- (3.85,2.125833024919768);
\draw (1.9,1.)-- (3.2,1.);
\draw (3.2,1.)-- (2.55,2.1258330249197694);
\draw (0.6,1.)-- (1.25,2.12583302491977);
\begin{scriptsize}
\draw [fill=black] (-2.,1.) circle (1.5pt);
\draw [fill=black] (-0.7,1.) circle (1.5pt);
\draw[color=black] (-1.3558677685950389,0.8) node {$e_2$};
\draw [fill=black] (-1.35,2.1258330249197703) circle (1.5pt);
\draw [fill=black] (-0.05,2.1258330249197703) circle (1.5pt);
\draw [fill=black] (0.6,1.) circle (1.5pt);
\draw [fill=black] (1.25,2.12583302491977) circle (1.5pt);
\draw [fill=black] (1.9,1.) circle (1.5pt);
\draw [fill=black] (2.55,2.1258330249197694) circle (1.5pt);
\draw [fill=black] (3.2,1.) circle (1.5pt);
\draw [fill=black] (3.85,2.125833024919768) circle (1.5pt);
\draw[color=black] (-0.7608264462809896,2.5024793388429756) node {$e_3$};
\draw[color=black] (-1.8847933884297494,1.8082644628099185) node {$e_1$};
\draw[color=black] (0.032561983471076236,0.8) node {$e_4$};
\draw[color=black] (3.9942148760330583,1.626446280991737) node {$e_{m+2}$};
\draw[color=black] (2.5614876033057863,0.8991735537190106) node {$e_m$};
\draw[color=black] (0.5284297520661173,2.4528925619834716) node {$e_5$};
\draw[dotted] (.6,1)-- (1.9,1);

\draw[dotted] (1.25,2.12583302491977)-- (2.55,2.1258330249197703);

\end{scriptsize}
\end{tikzpicture}

& $F_{2(m-1) + 2}$ \\ \hline 

$m + 2$ & 
  \begin{tikzpicture}[line cap=round,line join=round,>=triangle 45,x=1.0cm,y=1.0cm]
\draw (-2.,1.)-- (-0.7,1.);
\draw (-1.35,2.1258330249197703)-- (-2.,1.);
\draw (-1.35,2.1258330249197703)-- (-0.7,1.);
\draw (-0.05,2.1258330249197703)-- (-1.35,2.1258330249197703);

\draw (-0.05,2.1258330249197703)-- (-0.7,1.);
\draw (-0.7,1.)-- (0.6,1.);
\draw (0.6,1.)-- (-0.05,2.1258330249197703);
\draw (0.6,1.)-- (1.25,2.12583302491977);
\draw[dotted] (0.6,1.)-- (1.9,1);

\draw (1.25,2.12583302491977)-- (-0.05,2.1258330249197703);
\draw[dotted] (1.25,2.12583302491977)-- (2.55,2.1258330249197703);

\draw (2.55,2.1258330249197694)-- (1.9,1.);
\draw (1.9,1.)-- (3.2,1.);
\draw (3.2,1.)-- (2.55,2.1258330249197694);
\draw (2.55,2.1258330249197694)-- (3.85,2.125833024919768);
\begin{scriptsize}
\draw [fill=black] (-2.,1.) circle (1.5pt);
\draw [fill=black] (-0.7,1.) circle (1.5pt);
\draw[color=black] (-1.3558677685950398,0.8) node {$e_2$};
\draw [fill=black] (-1.35,2.1258330249197703) circle (1.5pt);
\draw[color=black] (-1.736033057851238,1.725619834710745) node {$e_1$};
\draw [fill=black] (-0.05,2.1258330249197703) circle (1.5pt);
\draw[color=black] (-0.6285950413223129,2.4) node {$e_3$};
\draw [fill=black] (0.6,1.) circle (1.5pt);
\draw[color=black] (-0.08314049586776764,0.8) node {$e_4$};
\draw [fill=black] (1.25,2.12583302491977) circle (1.5pt);
\draw [fill=black] (1.9,1.) circle (1.5pt);
\draw [fill=black] (2.55,2.1258330249197694) circle (1.5pt);
\draw [fill=black] (1.25,2.1258330249197703) circle (1.5pt);
\draw[color=black] (0.7102479338842981,2.4) node {$e_5$};
\draw[color=black] (2.412727272727273,0.8) node {$e_m$};
\draw [fill=black] (3.2,1.) circle (1.5pt);
\draw [fill=black] (3.85,2.125833024919768) circle (1.5pt);
\draw[color=black] (3.14,2.4) node {$e_{m+1}$};
\end{scriptsize}
\end{tikzpicture}

& $F_{2(m-1) + 2}$\\\hline 

\end{tabular}

\end{center}
\caption{Showing the number of spanning trees in $\mathscr{C}_k$ for $k \geq 2$ is the same as the number of spanning trees in $G_{k-2}$.}
\label{tab:tree}
\end{table}
\vspace{1cm}
\newpage
\section{Monotonicity, Minimal Resistance, and Link Prediction in the Straight Linear 2-tree}\label{sec:mono}
The first formula in Theorem~\ref{thm:main} gives
\begin{equation}\label{eq:mono1}
r_m(j,j+k+1)-r_m(j,j+k)=\frac{(F_{k+1}F_{2j+k-1}-F_kF_{2j+k-2})F_{2m-2j-2k+3}}{F_{2m+2}},\end{equation}
which for $k=1$ is
\[r_m(j,j+2)-r_m(j,j+1) = \frac{(F_{2j}-F_{2j-1})F_{2m-2j+1}}{F_{2m+2}}.\]
Since $j \leq m$, we have
\[r_m(j,j+2)-r_m(j,j+1) \begin{cases} =0 & \text {if $j=1$,}\\ > 0& \text{if $j > 1$.}\end{cases}\]
If $k \geq 2$ in~\eqref{eq:mono1}, $j+k+1 \leq m+2$, so $2m-2j+1 \geq 1$ and $F_{2m-2j+1} > 0.$ Then $r_m(j,j+k+1)-r_m(j,j+k) > 0.$
It follows that, 
\[
r_m(j,j+k+1)-r_m(j,j+k) : \left\{ \begin{array}{l}
=0 \text{ for } j = 1, \ k = 1\\
>0 \text{ otherwise.}
\end{array}\right.
\]
By the left-right symmetry of the linear 2-tree, 
\[
r_m(j-1,j+k-1)-r_m(j,j+k) : \left\{ \begin{array}{l}
=0 \text{ for } j = n-1, \ k = 1\\
>0 \text{ otherwise.}
\end{array}\right.
\]
This yields the following monotonicity result: 
\begin{theorem}\label{cor:monotonicity}
Let $1 \leq j' \leq j <j+k \leq j'+k' \leq n$, and let $m = n-2$.  Then 
\[
r_m(j',j'+k') - r_m(j,j+k) \geq 0
\] 
and equals 0 if and only if 
\begin{enumerate}
\item $j' = j$ and $k' = k$, 
\item $j' = j = 1$, $k = 1$, and $k' = 2$, or
\item $j' = n-2$, $j = n-1$, and $k = 1$,  $k' = 2$.
\end{enumerate}
\end{theorem}
\begin{corollary} The maximal resistance distance in a straight linear 2-tree on $n$ vertices is $r_m(1,n).$\end{corollary}
\begin{theorem}\label{thm:maxmin}
For a straight linear 2-tree with $n$ vertices and $m=n-2$ triangles, the function $g_{k,m}(j)=r_m(j,j+k)$ is unimodal for $j$ in $[1,n-k]$ with its maxima at the endpoints and minimum (minima) at 
\[
j=
\begin{cases}
\frac{m-k+3}2, & m-k \ {\rm odd},\\[2mm]
\frac{m-k+3+ 1}2, \frac{m-k+3- 1}2 &  m-k \ {\rm even.}
\end{cases}
\]

Moreover, it is strictly decreasing from the left endpoint to its (first) minimum, strictly increasing from its (second) minimum to its right endpoint, and $g_{k,m}(j)= g_{k,m}(n-k-j+1)$.
\end{theorem}
\begin{proof}Looking at the second formula in Theorem~\ref{thm:main}, we see that the only dependence on $j$ comes from the term $F_k^2 F_{m-2j-k+3}^2$.  Since $F_{m-2j-k+3}^2$ clearly has the stated properties and $F_k^2$ is constant and positive, so does $g_{k,m}$.
\end{proof}
\begin{corollary}\label{cor:minres}The minimal resistance distance in a straight linear 2-tree with $m$ triangles is  
\[
r_m\Big(\frac n2,\frac n2+1\Big)=\frac{F_{n-1}}{L_{n-1}}, \quad n \ {\rm even}\\[2mm]
\]
\[
r\Big(\frac{n-1}2,\frac{n+1}2\Big)=r\Big(\frac{n+1}2,\frac{n+3}2\Big)=\frac{F_{n-1}}{L_{n-1}}+\frac 1{F_{n-2}}, \quad n \ {\rm odd}.
\] 
\end{corollary}
\begin{proof}
By Theorem~\ref{cor:monotonicity}, the minimum must occur for an edge of the form $\{j,j+1\}$.  By the second formula in Theorem~\ref{thm:main}, 
\[
r_m(j,j+1)=\frac{F_{m+1}^2+F_{m-2j+2}^2}{F_{2m+2}}.
\]

By Theorem~\ref{thm:maxmin}, the minimum occurs at $j=\frac{m+2}2$ for $m$ even and for $j=(m+2\pm1)/2$ for $m$ odd.  The respective minima are 

\[
r\Big(\frac n2,\frac n2 +1\Big)= \frac{F_{n-1}^2+F_0^2}{F_{2n-2}} =\frac{F_{n-1}^2}{F_{n-1}L_{n-1}}=\frac{F_{n-1}}{L_{n-1}}
\]

\noindent and

\[
r\Big(\frac{n-1}2,\frac{n+1}2\Big)=r\Big(\frac{n+1}2,\frac{n+3}2\Big)=\frac{F_{n-1}^2+F_1^2}{F_{2n-2}}=\frac{F_{n-1}}{L_{n-1}}+\frac 1{F_{n-2}}. 
\] 

\noindent as claimed.
\end{proof}

Not surprisingly the minimal resistance distance occurs uniquely (for $n$ even) at the most centered edge $\left\{\frac{n}{2},\frac{n}{2}+1\right\}.$  For $n$ odd it occurs at the two most centered edges.  Less obvious is that since the minimal resistance distance,
\[ \frac{F_{n-1}}{L_{n-1}}\ \  ({\rm or} \ \frac{F_{n-1}}{L_{n-1}}+\frac 1{F_{n-2}})\ \  \rightarrow \frac{1}{\sqrt{5}}\]
as $n \rightarrow \infty$, every resistance distance in the sequence of straight linear 2-trees is bounded uniformly away from zero.  

When resistance distance is used in link prediction, the resistance distance is calculated for each non-edge and the non-edges are ranked from lowest to highest resistance distance.  If it is desired to predict $\ell$ links, then the $\ell$ non-edges with lowest resistance distance are the ones predicted.  It is evident from Theorem~\ref{cor:monotonicity} that the non-edge with lowest resistance distance must have the form $\{j, j+3\}$.  From the second formula in Theorem~\ref{thm:main},
\[r_m(j,j+3)=\frac{F^2_{m+1}+2F_{m+1}F_{m-1}+4F_{m-2j}^2}{F_{2m+2}}\]
which for $m$ even is minimized at $j = \frac{m}{2}=\frac{n}{2}-1$ by Theorem~\ref{thm:maxmin}. Then the non-edge with smallest resistance distance is $\{\frac{n}{2}-1,\frac{n}{2}+2\}.$  Note that $\frac{n}{2}-1$ and $\frac{n}{2}+2$ are the most central non-adjacent vertices in the straight linear 2-tree.  Beginning with unit resistances on every edge this is a natural choice for the first link predicted.  (A similar conclusion holds for $m$ odd.)

In fact, the next theorem will enable us to determine the ranking given by resistance distance of \textit{all} non-edges in a straight linear 2-tree. 

\begin{theorem}\label{thm:linearity}
$r_m(j,j+k) < r_m(\ell, \ell+(k+1))$ for every $\ell$ and $j$.
\end{theorem}

Before giving the proof of this theorem we state and prove the following.
\begin{lemma}\label{amanda}
\begin{multline}\label{eq:lhs1}
\frac{F_{m+1}}{5}
\left[
F_{m-k-1}((k+1)L_{k+1} - F_{k+1}) + F_{m-k}((k-4)F_{k+2} + (2k+4)F_{k+1})\right.\\
\hspace{5cm}\left.-F_{m-k}(kL_k-F_k) - F_{m-k+1}((k-5)F_{k+1} + (2k+2)F_k)
\right]\\ = F_{m+1}(F_{m-k+1}F_{k+1}-F_{k}F_{m-k}).
\end{multline}
\end{lemma}

\begin{proof}For ease of notation, define
\begin{multline*}LH(m,j,k)= \frac{F_{m+1}}{5}
\left[
F_{m-k-1}((k+1)L_{k+1} - F_{k+1}) + F_{m-k}((k-4)F_{k+2} + (2k+4)F_{k+1})\right.\\
\hspace{5cm}\left.-F_{m-k}(kL_k-F_k) - F_{m-k+1}((k-5)F_{k+1} + (2k+2)F_k)
\right].\end{multline*}
We expand $L_k = F_{k+1} + F_{k-1}$ to obtain
\[\begin{array}{c}
LH(m,j,k) = \frac{F_{m+1}}{5}
\left[
F_{m-k-1}((k+1)F_{k+2}+(k+1)F_{k} - F_{k+1}) - F_{m-k+1}((k-5)F_{k+1} + (2k+2)F_k)\right.\\[2mm]
\left.+F_{m-k}((k-4)F_{k+2} + (k+4)F_{k+1}-kF_{k-1}+F_k) 
\right]
\end{array}\]
Expanding the $F_{k+2}$ terms, this becomes

\[\begin{array}{c}
\frac{F_{m+1}}{5}
\left[
F_{m-k-1}(kF_{k+1}+(2k+2)F_{k}) - F_{m-k+1}((k-5)F_{k+1} + (2k+2)F_k)\right.\\[2mm]
\left.+F_{m-k}(2kF_{k+1} +(k-3)F_{k}  -kF_{k-1}) 
\right]
\end{array}\]
Next, $F_{m-k+1} = F_{m-k} + F_{m-k-1}$ gives
\[\begin{array}{c}
\frac{F_{m+1}}{5}
\left[
F_{m-k-1}(5F_{k+1})+F_{m-k}((k+5)F_{k+1} -(k+5)F_{k}  -kF_{k-1}) 
\right]\\[2mm]
=\frac{F_{m+1}}{5}
\left[
F_{m-k-1}(5F_{k+1})+F_{m-k}(5F_{k-1}  ) 
\right]
\end{array}\]
Finally, we note that $F_{k-1} = F_{k+1} - F_k$ to obtain \eqref{eq:lhs1}.

\end{proof}
\begin{proof} (of Theorem~\ref{thm:linearity})
By Theorem~\ref{thm:maxmin}, $r_m(1,1+k) > r_m(j,j+k)$ for all $j$ with $1<j<n-k$.  We first consider the case when $m-(k+1)$ is odd.  In this case the minimum of $r_m(\ell, \ell+(k+1))$ occurs at $\ell =(m-(k+1)+3)/2= (n-(k+1)+1)/2$ , so it suffices to show that 
\[r_m(1,1+k) < r_m\left(\frac{n-k}{2}, \frac{n-k}{2}+k+1\right).\]
Using Theorem~\ref{thm:main} and Lemma~\ref{amanda}
\[\begin{aligned}
&r_m((n-k)/2, (n-k)/2+(k+1))-r_m(1,1+k) \\[2mm]
&=F_{m+1}(F_{k+1}F_{m-k+1} - F_kF_{m-k})-F_k^2F_{m-k+1}^2 \\[2mm]
&=(F_{k+1}F_{m-k+1} + F_kF_{m-k})(F_{k+1}F_{m-k+1} - F_kF_{m-k})-F_k^2F_{m-k+1}^2 \\[2mm]
&=F_{k+1}^2F_{m-k+1}^2-F_k^2F_{m-k}^2-F_k^2F_{m-k+1}^2 \\[2mm]
&=(F_{k+1}^2-F_k^2)F_{m-k+1}^2-F_k^2F_{m-k}^2 \\[2mm]
&=(F_{k+1}+F_k)(F_{k+1}-F_k)F_{m-k+1}^2-F_k^2F_{m-k}^2 \\[2mm]
&=F_{k+2}F_{k-1}F_{m-k+1}^2-F_k^2F_{m-k}^2 \\[2mm]
&=(F_{k+2}-F_{k+1})F_{k-1}F_{m-k+1}^2+F_{k+1}F_{k-1}F_{m-k+1}^2-F_k^2F_{m-k}^2 \\[2mm]
&=F_kF_{k-1}F_{m-k+1}^2+[F_k^2+(-1)^k]F_{m-k+1}^2-F_k^2F_{m-k}^2 \\[2mm]
&=[F_kF_{k-1}+(-1)^k]F_{m-k+1}^2+F_k^2(F_{m-k+1}^2-F_{m-k}^2)\\[2mm]
&=[F_kF_{k-1}+(-1)^k]F_{m-k+1}^2+F_k^2F_{m-k+2}F_{m-k-1}.
\end{aligned}\]

The first term is greater than $0$ for all $k \ge 2$ and the second term is greater than or equal to $0$ for all $k \ge 2$ and equals 0 if and only if $m=k+1$.
A similar argument can be used to show that $r_m(1,1+k) < r_m(\ell,\ell+k)$ for all $\ell$ in the case that $m-(k+1)$ is even.
\end{proof}
Using Theorem~\ref{thm:maxmin} and Theorem~\ref{thm:linearity} together the non-edges of a straight linear 2-tree can easily be put in order according to resistance distance.  One begins with the list of non-edges with nodes that are three apart selecting first the center-most non-edge(s) and moving toward the ends in both directions. Next one considers the list of non-edges with nodes four apart selecting again the center-most non-edge and moving towards the ends in both directions.  One continues in this manner until selecting the last non-edge $\{1,n\}$.

For example, given the straight linear 2-tree on 9 vertices, the non-edges ranked from smallest to greatest resistance distance, with an \& sign if they are tied, are: 
 
\[\{3,6\} \& \{4,7\}, \{2,5\} \& \{5,8\}, \{1,4\} \& \{6,9\}, \{3,7\}, \{2,6\} \& \{4,8\}, \{1,5\} \& \{5,9\}, \ldots, \{1,8\} \& \{2,9\}, \{1,9\}.\]
In order to predict $\ell$ links simply take the first $\ell$ non-edges on the list, breaking a tie arbitrarily if necessary. 

Other than trees and cycles, these linear 2-trees are the only arbitrarily large graphs we know of for which all non-edges can be totally ordered by resistance distance.  Of course, this infinite family we have considered is highly structured.  It would be interesting to find regularity in the ranked predicted links for less structured graphs.  Nevertheless, it is now evident that  the extremely regular behavior of resistance distance in straight linear 2-trees shares nothing in common with that of the random geometric graphs in~\cite{lostinspace}.  We believe many other graph classes exhibit this kind of regularity.
\section{Conclusions and conjectures}
Theorem~\ref{cor:onefifth} implies that if $G$ is a straight linear 2-tree with $n$ vertices, then $r_G(1,n)\rightarrow \infty$  as $n\rightarrow \infty$.  Rayleigh's monotonicity Law (see for example~\cite[Lemma D]{KleinRandic}) states that if $H$ is an electrical circuit with a resistance on each edge and we create a new circuit $K$ from $H$ by lowering the resistance on an existing edge or inserting a new edge, then $r_K(i,j) \leq r_H(i,j)$ for
 all vertices $i$ and $j$ of $H$. This fact coupled with Theorem~\ref{cor:onefifth} imply for any connected subgraph of this straight linear 2-tree containing vertices $1$ and $n$, that $r(1,n)\rightarrow \infty$ also as $n\rightarrow\infty$.   This is the opposite behavior from~\cite{lostinspace} which would say (by Equation~\ref{eqn_LIS}) that $r(1,n)\rightarrow 1$.

\begin{definition}
A straight linear $k$-tree is a graph $G_n$ with $n$ vertices with adjacency matrix that is symmetric, banded, with the first through $k$th subdiagonals equal to one and first through $k$th superdiagonals equal to one, and all other entries equal to zero.  In other words $G_n$ is the graph whose (0,1)-adjacency matrix is defined by 
\[
a_{ij}=\begin{cases}1 &\text{if $0 < |i-j| \leq k$}\\0 & \text{otherwise.}\end{cases}\]
\end{definition}
Empirical evidence has allowed us to make the following conjecture:
\begin{conjecture}
	Let $G$ be the straight linear k-tree, $k \geq 1$, with $n$ vertices and $H$ be the straight linear k-tree with $n+1$ vertices.
	Then 
	\[\lim_{n\rightarrow \infty} r_{H} (1, n+1) - r_G(1,n) = \frac{6}{k(k+1)(2k+1)}.\]
	\end{conjecture}
	
We have also done work on the broader class of linear 2-trees~\cite{bef2}.
We recall that a $k$-tree is constructed inductively by starting with a complete graph on $k + 1$ vertices and connecting each new vertex to the vertices of an existing clique on $k$ vertices. 
For example, the five 2-trees on 6 vertices are given in Figure \ref{fig:0}. 

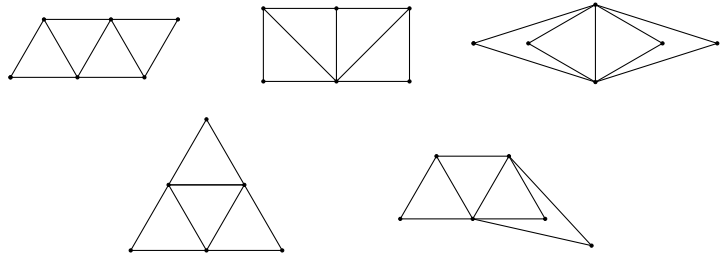
\begin{figure}[ht]
\begin{center}
\begin{tikzpicture}[scale = .3, line cap=round,line join=round,>=triangle 45,x=1.0cm,y=1.0cm]
\draw (-1.3569135281606972,0.6673644767836225)-- (0.1270096766510902,3.2375948620480752);
\draw (0.1270096766510902,3.2375948620480752)-- (1.6109328814628765,0.6673644767836225);
\draw (1.6109328814628765,0.6673644767836225)-- (3.094856086274665,3.237594862048074);
\draw (3.094856086274665,3.237594862048074)-- (4.578779291086451,0.6673644767836205);
\draw (4.578779291086451,0.6673644767836205)-- (6.06270249589824,3.2375948620480717);
\draw (6.06270249589824,3.2375948620480717)-- (3.094856086274665,3.237594862048074);
\draw (3.094856086274665,3.237594862048074)-- (0.1270096766510902,3.2375948620480752);
\draw (-1.3569135281606972,0.6673644767836225)-- (1.6109328814628765,0.6673644767836225);
\draw (1.6109328814628765,0.6673644767836225)-- (4.578779291086451,0.6673644767836205);
\draw (9.84889663182346,3.7251335656213667)-- (9.84889663182346,0.4872241579558654);
\draw (9.84889663182346,0.4872241579558654)-- (13.08680603948896,0.4872241579558653);
\draw (13.08680603948896,0.4872241579558653)-- (13.08680603948896,3.725133565621366);
\draw (13.08680603948896,3.725133565621366)-- (9.84889663182346,3.7251335656213667);
\draw (13.08680603948896,3.725133565621366)-- (16.32471544715446,3.725133565621366);
\draw (16.32471544715446,3.725133565621366)-- (16.32471544715446,0.4872241579558654);
\draw (16.32471544715446,0.4872241579558654)-- (13.08680603948896,0.4872241579558653);
\draw (9.84889663182346,3.7251335656213667)-- (13.08680603948896,0.4872241579558653);
\draw (13.08680603948896,0.4872241579558653)-- (16.32471544715446,3.725133565621366);
\draw (24.563559718763077,3.891443382031552)-- (24.563559718763077,0.458111183044603);
\draw (24.563559718763077,3.891443382031552)-- (21.59020681480929,2.174777282538077);
\draw (21.59020681480929,2.174777282538077)-- (24.563559718763077,0.458111183044603);
\draw (24.563559718763077,0.458111183044603)-- (27.536912622716866,2.174777282538078);
\draw (27.536912622716866,2.174777282538078)-- (24.563559718763077,3.891443382031552);
\draw (24.563559718763077,3.891443382031552)-- (19.163864533083597,2.1747772825380767);
\draw (19.163864533083597,2.1747772825380767)-- (24.563559718763077,0.458111183044603);
\draw (24.563559718763077,3.891443382031552)-- (29.963254904442575,2.174777282538078);
\draw (29.963254904442575,2.174777282538078)-- (24.563559718763077,0.458111183044603);
\draw (9.006222295606864,-4.100651066546037)-- (5.649150127344132,-4.100651066546037);
\draw (9.006222295606864,-4.100651066546037)-- (5.649150127344132,-4.100651066546037);
\draw (7.327686211475499,-1.193341286492803)-- (9.006222295606864,-4.100651066546037);
\draw (7.327686211475499,-1.193341286492803)-- (5.649150127344132,-4.100651066546037);
\draw (5.649150127344132,-4.100651066546037)-- (3.9706140432127657,-7.0079608465992695);
\draw (3.9706140432127657,-7.0079608465992695)-- (7.3276862114754975,-7.007960846599271);
\draw (7.3276862114754975,-7.007960846599271)-- (5.649150127344132,-4.100651066546037);
\draw (7.3276862114754975,-7.007960846599271)-- (9.006222295606864,-4.100651066546037);
\draw (9.006222295606864,-4.100651066546037)-- (10.68475837973823,-7.007960846599272);
\draw (10.68475837973823,-7.007960846599272)-- (7.3276862114754975,-7.007960846599271);
\draw (20.731648274611473,-2.8279958742653273)-- (24.392377272176738,-6.795455417778197);
\draw (24.392377272176738,-6.795455417778197)-- (19.12609202022495,-5.608900881272767);
\draw (17.520535765838424,-2.827995874265328)-- (15.914979511451902,-5.608900881272767);
\draw (15.914979511451902,-5.608900881272767)-- (19.12609202022495,-5.608900881272767);
\draw (19.12609202022495,-5.608900881272767)-- (17.520535765838424,-2.827995874265328);
\draw (17.520535765838424,-2.827995874265328)-- (20.731648274611473,-2.8279958742653273);
\draw (20.731648274611473,-2.8279958742653273)-- (19.12609202022495,-5.608900881272767);
\draw (19.12609202022495,-5.608900881272767)-- (22.33720452899799,-5.608900881272771);
\draw (22.33720452899799,-5.608900881272771)-- (20.731648274611473,-2.8279958742653273);
\begin{scriptsize}
\draw [fill=black] (-1.3569135281606972,0.6673644767836225) circle (2.0pt);
\draw [fill=black] (1.6109328814628765,0.6673644767836225) circle (2.0pt);
\draw [fill=black] (0.1270096766510902,3.2375948620480752) circle (2.0pt);
\draw [fill=black] (-1.3569135281606983,0.6673644767836237) circle (2.0pt);
\draw [fill=black] (3.094856086274665,3.237594862048074) circle (2.0pt);
\draw [fill=black] (4.578779291086451,0.6673644767836205) circle (2.0pt);
\draw [fill=black] (6.06270249589824,3.2375948620480717) circle (2.0pt);
\draw [fill=black] (13.08680603948896,0.4872241579558653) circle (2.0pt);
\draw [fill=black] (13.08680603948896,3.725133565621366) circle (2.0pt);
\draw [fill=black] (9.84889663182346,3.7251335656213667) circle (2.0pt);
\draw [fill=black] (9.84889663182346,0.4872241579558654) circle (2.0pt);
\draw [fill=black] (16.32471544715446,0.4872241579558654) circle (2.0pt);
\draw [fill=black] (16.32471544715446,3.725133565621366) circle (2.0pt);
\draw [fill=black] (24.563559718763077,3.891443382031552) circle (2.0pt);
\draw [fill=black] (24.563559718763077,0.458111183044603) circle (2.0pt);
\draw [fill=black] (27.536912622716866,2.174777282538078) circle (2.0pt);
\draw [fill=black] (21.59020681480929,2.174777282538077) circle (2.0pt);
\draw [fill=black] (19.163864533083597,2.1747772825380767) circle (2.0pt);
\draw [fill=black] (19.163864533083597,2.1747772825380767) circle (2.0pt);
\draw [fill=black] (29.963254904442575,2.174777282538078) circle (2.0pt);
\draw [fill=black] (9.006222295606864,-4.100651066546037) circle (2.0pt);
\draw [fill=black] (5.649150127344132,-4.100651066546037) circle (2.0pt);
\draw [fill=black] (7.327686211475499,-1.193341286492803) circle (2.0pt);
\draw [fill=black] (7.3276862114754975,-7.007960846599271) circle (2.0pt);
\draw [fill=black] (3.9706140432127657,-7.0079608465992695) circle (2.0pt);
\draw [fill=black] (10.68475837973823,-7.007960846599272) circle (2.0pt);
\draw [fill=black] (15.914979511451902,-5.608900881272767) circle (2.0pt);
\draw [fill=black] (19.12609202022495,-5.608900881272767) circle (2.0pt);
\draw [fill=black] (17.520535765838424,-2.827995874265328) circle (2.0pt);
\draw [fill=black] (20.731648274611473,-2.8279958742653273) circle (2.0pt);
\draw [fill=black] (22.33720452899799,-5.608900881272771) circle (2.0pt);
\draw [fill=black] (24.392377272176738,-6.795455417778197) circle (2.0pt);
\end{scriptsize}
\end{tikzpicture}
\caption{All the 2-trees on 6 vertices} \label{fig:0}
\end{center}
\end{figure}

\begin{definition}[linear 2-tree]\label{def:lin2tree}
A linear 2-tree is a graph $G$ that is constructed inductively by starting with a triangle and connecting each new vertex to the vertices of an existing edge that includes a vertex of degree 2.  In other words, a linear 2-tree is a 2-tree with exactly two vertices of degree 2. \end{definition}

So just two of the graphs in Figure \ref{fig:0} are linear 2-trees.  It is straightforward to modify the argument in Section \ref{sec:spanningtrees} to obtain the following theorem. 

\begin{theorem}
The number of spanning trees in a linear 2-tree with $m$ triangles is $F_{2m+2}$. 
\end{theorem}


\begin{definition}[bent linear 2-tree]\label{def:lin2treestb}
A bent linear 2-tree is a graph with $n$ vertices such that for some $k \in \{3,4, \ldots, n-3\}$, its adjacency matrix is
\[a_{ij} = \begin{cases} 1 & \text{if $j = i+1$}\\
1 & \text{if $j = i-1$}\\
1 & \text{if $j = i+2$ and $ i \neq k+1$}\\
1 & \text{if $j = i-2$ and $i \neq k+3$}\\
1 & \text{if $j = i+3$ and $i = k$}\\
1 & \text{if $j = i-3$ and $i = k+3$}\\
0 & \text{otherwise.}\end{cases}\]
 See Figure~\ref{fig:bent}.
\end{definition}
We have the following result for a linear 2-tree that is straight except for the existence of a bend at vertex $k$ (see Figure~\ref{fig:bent}).
\begin{theorem}
Given a bent linear 2-tree with $n$ vertices, $m = n-2$ cells, and one bend located at vertex $k$, the resistance distance between node 1 and node $n$ is
	 \begin{equation}\label{eq:univconf}r_{m,k}(1,n)=\frac{m+1}{5}+\frac{4F_{m+1}}{5L_{m+1}}\frac{\sum_{j=3}^{k}\left[(-1)^jF_{m-2j+3}(F_{m+2}+F_{j-2}F_{m-j+1})\right]}{F_{2m+2}}.\end{equation}
	 \end{theorem}
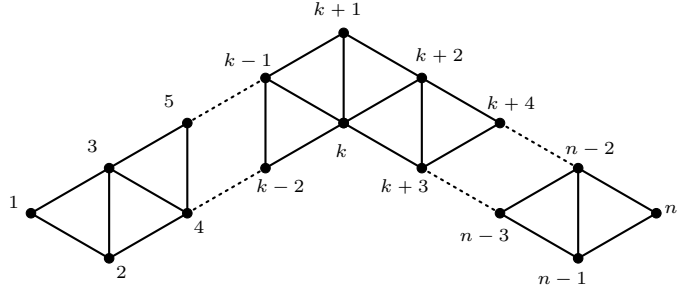
\begin{figure}
\begin{center}
\begin{tikzpicture}[line cap=round,line join=round,>=triangle 45,x=1.0cm,y=1.0cm,scale = 1.2]
\draw [line width=.8pt] (-5.464101615137757,2.)-- (-4.598076211353318,1.5);
\draw [line width=.8pt] (-4.598076211353318,1.5)-- (-4.598076211353318,2.5);
\draw [line width=.8pt] (-4.598076211353318,2.5)-- (-3.732050807568879,2.);
\draw [line width=.8pt] (-3.732050807568879,2.)-- (-3.7320508075688785,3.);
\draw [line width=.8pt] (-2.8660254037844393,2.5)-- (-2.866025403784439,3.5);
\draw [line width=.8pt] (-2.866025403784439,3.5)-- (-2.,3.);
\draw [line width=.8pt] (-2.,3.)-- (-2.,4.);
\draw [line width=.8pt] (-2.,4.)-- (-1.1339745962155612,3.5);
\draw [line width=.8pt] (-1.1339745962155612,3.5)-- (-2.,3.);
\draw [line width=.8pt] (-2.,3.)-- (-1.1339745962155616,2.5);
\draw [line width=.8pt] (-1.1339745962155616,2.5)-- (-1.1339745962155612,3.5);
\draw [line width=.8pt] (-1.1339745962155612,3.5)-- (-0.2679491924311225,3.);
\draw [line width=.8pt] (-0.2679491924311225,3.)-- (-1.1339745962155616,2.5);
\draw [line width=.8pt,dotted] (-1.1339745962155616,2.5)-- (-0.26794919243112303,2.);
\draw [line width=.8pt,dotted] (-0.2679491924311225,3.)-- (0.5980762113533165,2.5);
\draw [line width=.8pt] (0.5980762113533165,2.5)-- (-0.26794919243112303,2.);
\draw [line width=.8pt] (-0.26794919243112303,2.)-- (0.5980762113533152,1.5);
\draw [line width=.8pt] (0.5980762113533152,1.5)-- (0.5980762113533165,2.5);
\draw [line width=.8pt] (0.5980762113533165,2.5)-- (1.464101615137755,2.);
\draw [line width=.8pt] (1.464101615137755,2.)-- (0.5980762113533152,1.5);
\draw [line width=.8pt] (-2.,4.)-- (-2.866025403784439,3.5);
\draw [line width=.8pt,dotted] (-2.866025403784439,3.5)-- (-3.7320508075688785,3.);
\draw [line width=.8pt] (-3.7320508075688785,3.)-- (-4.598076211353318,2.5);
\draw [line width=.8pt] (-4.598076211353318,2.5)-- (-5.464101615137757,2.);
\draw [line width=.8pt] (-4.598076211353318,1.5)-- (-3.732050807568879,2.);
\draw [line width=.8pt,dotted] (-3.732050807568879,2.)-- (-2.8660254037844393,2.5);
\draw [line width=.8pt] (-2.8660254037844393,2.5)-- (-2.,3.);
\begin{scriptsize}
\draw [fill=black] (-2.,4.) circle (1.5pt);
\draw[color=black] (-2.06648828953202,4.259331085745072) node {$k+1$};
\draw [fill=black] (-1.1339745962155612,3.5) circle (1.5pt);
\draw[color=black] (-0.9407359844212153,3.7428094398707032) node {$k+2$};
\draw [fill=black] (-2.,3.) circle (1.5pt);
\draw[color=black] (-2.0267558552339913,2.6989231016718021) node {$k$};
\draw [fill=black] (-2.866025403784439,3.5) circle (1.5pt);
\draw[color=black] (-3.0597991469827295,3.689832860806665) node {$k-1$};
\draw [fill=black] (-2.8660254037844393,2.5) circle (1.5pt);
\draw[color=black] (-2.702207238300474,2.2919066737377372) node {$k-2$};
\draw [fill=black] (-1.1339745962155616,2.5) circle (1.5pt);
\draw[color=black] (-1.3248161826354898,2.2919066737377372) node {$k+3$};
\draw [fill=black] (-2.,3.) circle (1.5pt);
\draw [fill=black] (-0.2679491924311225,3.) circle (1.5pt);
\draw[color=black] (-0.14608729846064736,3.213043649230324) node {$k+4$};
\draw [fill=black] (-3.7320508075688785,3.) circle (1.5pt);
\draw[color=black] (-3.9339127015393545,3.2395319387623434) node {$5$};
\draw [fill=black] (-3.732050807568879,2.) circle (1.5pt);
\draw[color=black] (-3.6028090823891175,1.8488967383313488) node {$4$};
\draw [fill=black] (-0.26794919243112303,2.) circle (1.5pt);
\draw[color=black] (-0.4374584833128556,1.7488967383313488) node {$n-3$};
\draw [fill=black] (-4.598076211353318,2.5) circle (1.5pt);
\draw[color=black] (-4.78153796656396,2.7494985824199927) node {$3$};
\draw [fill=black] (-4.598076211353318,1.5) circle (1.5pt);
\draw[color=black] (-4.463678492179733,1.345619237222989) node {$2$};
\draw [fill=black] (-5.464101615137757,2.) circle (1.5pt);
\draw[color=black] (-5.655651521120585,2.1270237784175476) node {$1$};
\draw [fill=black] (0.5980762113533165,2.5) circle (1.5pt);
\draw[color=black] (0.7280262560959774,2.709766148121964) node {$n-2$};
\draw [fill=black] (0.5980762113533152,1.5) circle (1.5pt);
\draw[color=black] (0.4234109264777597,1.2721075267550078) node {$n-1$};
\draw [fill=black] (1.464101615137755,2.) circle (1.5pt);
\draw[color=black] (1.628628100184621,2.0475589098214906) node {$n$};
\end{scriptsize}
\end{tikzpicture}
\end{center}
\caption{A linear 2-tree with $n$ vertices and single bend at vertex $k$.}
\label{fig:bent}
\end{figure}

With this result in mind we have the following conjecture
\begin{conjecture}
Let $G$ be a linear 2-tree.  If the diameter of the 2-tree tends to infinity, the maximal resistance distance of the linear 2-tree is also unbounded.
\end{conjecture}	 
A graph has tree width $\leq k$ if it is the subgraph of some $k$-tree. 
We have 
done computer calculations for the triangular grid graph on $3, 6, 10, 15, 21, \ldots$ vertices shown in Figure~\ref{fig:supertriangle} (for 15 vertices) which does not have bounded tree-width as the number of vertices goes to infinity.  Nevertheless, assuming unit resistances on each edge, we have the following conjecture about the growth of the resistance distance.
\begin{figure}
\begin{center}
\begin{tikzpicture}[scale = .6, line cap=round,line join=round,>=triangle 45,x=1.0cm,y=1.0cm,scale = 1.2]
\draw [line width=.8pt,color=black] (-2.,10.)-- (-3.,11.732050807568879);
\draw [line width=.8pt,color=black] (-2.,10.)-- (-4.,10);
\draw [line width=.8pt,color=black] (-3.,11.732050807568879)-- (-4.,10.);
\draw [line width=.8pt,color=black] (-4.,10.)-- (-3.,8.267949192431121);
\draw [line width=.8pt,color=black] (-3.,8.267949192431121)-- (-2.,10.);
\draw [line width=.8pt,color=black] (-4.,10.)-- (-5.,8.267949192431123);
\draw [line width=.8pt,color=black] (-5.,8.267949192431123)-- (-3.,8.267949192431121);
\draw [line width=.8pt,color=black] (-3.,8.267949192431121)-- (-1.,8.267949192431121);
\draw [line width=.8pt,color=black] (-1.,8.267949192431121)-- (-2.,10.);
\draw [line width=.8pt,color=black] (-3.,8.267949192431121)-- (-2.,6.535898384862243);
\draw [line width=.8pt,color=black] (-3.,8.267949192431121)-- (-4.,6.535898384862244);
\draw [line width=.8pt,color=black] (-1.,8.267949192431121)-- (-2.,6.535898384862244);
\draw [line width=.8pt,color=black] (-4.,6.535898384862244)-- (-2.,6.535898384862243);
\draw [line width=.8pt,color=black] (-3.,8.267949192431121)-- (-5.,8.267949192431121);
\draw [line width=.8pt,color=black] (-5.,8.267949192431121)-- (-4.,6.535898384862244);
\draw [line width=.8pt,color=black] (-2.,6.535898384862243)-- (0.,6.535898384862242);
\draw [line width=.8pt,color=black] (0.,6.535898384862242)-- (-1.,8.267949192431121);
\draw [line width=.8pt,color=black] (-5.,8.267949192431123)-- (-6.,6.535898384862246);
\draw [line width=.8pt,color=black] (-6.,6.535898384862246)-- (-4.,6.535898384862244);
\draw [line width=.8pt,color=black] (-2.,6.535898384862243)-- (-1.,4.803847577293364);
\draw [line width=.8pt,color=black] (-1.,4.803847577293364)-- (0.,6.535898384862242);
\draw [line width=.8pt,color=black] (-4.,6.535898384862244)-- (-3.,4.8038475772933635);
\draw [line width=.8pt,color=black] (-3.,4.8038475772933635)-- (-2.,6.535898384862243);
\draw [line width=.8pt,color=black] (-6.,6.535898384862246)-- (-5.,4.803847577293364);
\draw [line width=.8pt,color=black] (-5.,4.803847577293364)-- (-4.,6.535898384862244);
\draw [line width=.8pt,color=black] (-6.,6.535898384862246)-- (-7.,4.80384757729337);
\draw [line width=.8pt,color=black] (1.,4.803847577293362)-- (0.,6.535898384862242);
\draw [line width=.8pt,color=black] (-7.,4.80384757729337)-- (-5.,4.803847577293364);
\draw [line width=.8pt,color=black] (-3.,4.80384757729337)-- (-5.,4.803847577293364);
\draw [line width=.8pt,color=black] (-3.,4.80384757729337)-- (-1.,4.803847577293364);
\draw [line width=.8pt,color=black] (-1.,4.803847577293364)-- (1.,4.803847577293362);
\draw [line width=.8pt,color=black] (-5.,4.803847577293364)-- (-7.,4.80384757729337);

\begin{scriptsize}
\draw [fill=black] (-4.,10.) circle (2.5pt);
\draw[color=black] (-3,12.1) node {$a$};
\draw [fill=black] (-2.,10.) circle (2.5pt);
\draw [fill=black] (-3.,11.732050807568879) circle (2.5pt);
\draw [fill=black] (-3.,8.267949192431121) circle (2.5pt);
\draw [fill=black] (-5.,8.267949192431123) circle (2.5pt);
\draw [fill=black] (-1.,8.267949192431121) circle (2.5pt);
\draw [fill=black] (-2.,6.535898384862243) circle (2.5pt);
\draw [fill=black] (-4.,6.535898384862244) circle (2.5pt);
\draw [fill=black] (-5.,8.267949192431121) circle (2.5pt);
\draw [fill=black] (0.,6.535898384862242) circle (2.5pt);
\draw [fill=black] (-6.,6.535898384862246) circle (2.5pt);
\draw [fill=black] (-1.,4.803847577293364) circle (2.5pt);
\draw [fill=black] (-3.,4.8038475772933635) circle (2.5pt);
\draw [fill=black] (-5.,4.803847577293364) circle (2.5pt);
\draw [fill=black] (-7.,4.80384757729337) circle (2.5pt);
\draw [fill=black] (1.,4.803847577293362) circle (2.5pt);
\draw[color=black] (-7.3,4.5) node {$b$};

\end{scriptsize}
\end{tikzpicture}
\end{center}
\caption{A triangular grid with $4$ rows.}
\label{fig:supertriangle}
\end{figure}
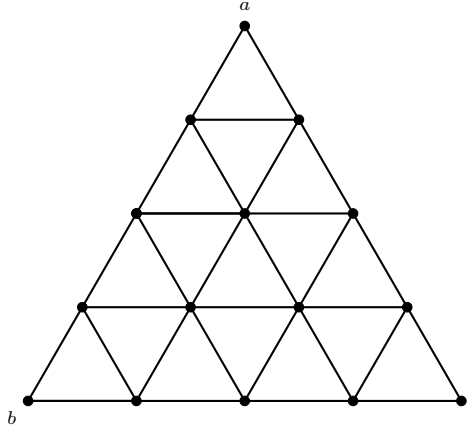

 \begin{conjecture}
Let $T_n$ be the triangular grid graph shown in Figure~\ref{fig:supertriangle} with $n$ rows and $m=n^2$ cells.  Moreover let $a$ and $b$ be distinct vertices with degree 2 and let $r_n(a,b)$ be the resistance distance between $a$ and $b$ in $T_n$.  Then
\[\lim_{n\rightarrow \infty} \exp(r_{n+1}(a,b)) -\exp(r_n(a,b)) = C > 0.\]
Moreover $\lim_{n\rightarrow \infty} r_n(a,b) = \infty.$
\end{conjecture}

These examples convincingly demonstrate that resistance distance should not be discounted as a method for link prediction.  A difficult, subtle problem is to make a judgment about a given real-life network as to whether it is more like a random geometric graph, a subgraph of a linear $k$-tree, a subgraph of a triangular grid, or in some other class.  

\begin{appendices}
\section{Appendix}\label{app}

The following two propositions complete the proof of Theorem \ref{thm:effRes}
\begin{proposition}\label{prop:rm1nA}
For $m \geq 1$, 
\[
 \sum_{i = 1}^{m} \frac{F_i F_{i+1}}{L_i L_{i+1}} = \frac{(m+1) L_{m+1} - F_{m+1}}{5 L_{m+1}}.
\]
\end{proposition}

\begin{proof}
For $m = 1$: 
\[
\frac{2 \cdot L_2 - F_2}{5L_2} = \frac{2\cdot 3 - 1}{5 \cdot 3} = \frac{5}{15} = \frac{1}{3},
\text{ and }
 \frac{F_1 F_{2}}{L_1 L_{2}} =  \frac{1}{3}.
\]
Define
\[S_m= \sum_{i = 1}^{m} \frac{F_i F_{i+1}}{L_i L_{i+1}}\]
and assume that 
\[
 S_m = \frac{(m+1) L_{m+1} - F_{m+1}}{5 L_{m+1}}.
\]
Then, 
\[
S_{m+1}=\sum_{i = 1}^{m+1} \frac{F_i F_{i+1}}{L_i L_{i+1}} = \frac{(m+1) L_{m+1} - F_{m+1}}{5 L_{m+1}} + \frac{F_{m+1}F_{m+2}}{L_{m+1}L_{m+2}}
= \frac{(m+1) L_{m+1}L_{m+2} - F_{m+1}L_{m+2}+ 5F_{m+1}F_{m+2}}{5L_{m+1}L_{m+2}}.
\]
Applying Proposition~\ref{prop:15sum} yields, 

\[\begin{aligned}
S_{m+1} &= \frac{(m+1) L_{m+1}L_{m+2} - F_{m+1}L_{m+2}+ F_{m+1}L_{m+1} + F_{m+1} L_{m+3}}{5L_{m+1}L_{m+2}}.\\
& = \frac{(m+1) L_{m+1}L_{m+2} + 2F_{m+1}L_{m+1}}{5L_{m+1}L_{m+2}}
= \frac{(m+1) L_{m+2} + 2F_{m+1}}{5L_{m+2}}\\
&=  \frac{(m+2) L_{m+2} -F_{m+2}}{5L_{m+2}}.
\end{aligned}\]
Since $2F_{m+1} = F_{m+1} + F_{m+3}  -F_{m+2} = L_{m+2} - F_{m+2}$ by Corollary~\ref{cor:2fm1}.\end{proof}

\begin{proposition}\label{prop:rm1nB}
\[
\frac{2F_{m+1}^2}{L_m L_{m+1}} + \frac{mL_{m} - F_{m}}{5 L_{m}}= \frac{m+1}{5} +  \frac{4F_{m+1}}{5L_{m+1}}.
\]
\end{proposition}
\begin{proof}
Using Proposition~\ref{prop:15sum} we obtain
 \begin{multline*}
    \frac{2F_{m+1}^2}{L_m L_{m+1}} + \frac{m L_{m} - F_{m}}{5 L_{m}}  \\
=   \frac{10F_{m+1}^2 + m L_{m}L_{m+1} - F_{m}L_{m+1}}{5L_m L_{m+1}} 
=\frac{2F_{m+1}(L_m+L_{m+2})+ m L_{m}L_{m+1} - F_{m}L_{m+1}}{5L_m L_{m+1}} \\
=   \frac{L_m(2F_{m+1} + m L_{m+1}) + 2F_{m+1}L_{m+2} - F_{m}L_{m+1}}{5L_m L_{m+1}} 
=   \frac{L_m(2F_{m+1} + m L_{m+1}) + 2F_{m+1}(L_{m}+L_{m+1}) - F_{m}L_{m+1}}{5L_m L_{m+1}} \\
  = \frac{L_m(4F_{m+1} + m L_{m+1}) + 2F_{m+1}L_{m+1} - F_{m}L_{m+1}}{5L_m L_{m+1}}
 = \frac{L_m(4F_{m+1} + m L_{m+1}) + L_{m+1}(2F_{m+1} - F_{m})}{5L_m L_{m+1}} \\
 =    \frac{L_m(4F_{m+1} + m L_{m+1}) + L_{m+1}(L_m)}{5L_m L_{m+1}} 
   = \frac{4F_{m+1} + (m+1) L_{m+1}}{5 L_{m+1}} \\
 = \frac{m+1}{5} +  \frac{4F_{m+1}}{5L_{m+1}}.  
\end{multline*}
\end{proof}

The following two propositions complete the proof of Theorem \ref{thm:main}
\begin{proposition}\label{app:simpA}
\begin{equation}\label{eq:appA}
F_{k+1}F_{m-2j-k+2}^2 + F_{m+1}F_{m-k+1}=F_{2m-2j-2k+3}F_{2j+k-1}+F_kF_{m-2j-k+2}F_{m-2j-k+3}
\end{equation}
\end{proposition}
\begin{proof}
Let $A$ be the left-hand side of~\eqref{eq:appA}.  We transform $A$ using Proposition~\ref{prop:splitsum} and the relations $F_{-n}=(-1)^{n+1}F_n$, $F_{m-2j-k+2}F_{m-k} +F_{m-2j-k+3}F_{m-k+1} = F_{2m-2j-2k+3}$ and $F_{2j+k-1} = F_{-2j+1}F_{k+1} - F_{k}F_{-2j+2}$. 
\[\begin{aligned}
A &= F_{k+1}F_{m-2j-k+2}(F_{m-k}F_{-2j+1}+F_{m-k+1}F_{-2j+2}) + F_{m+1}F_{m-k+1}\\[2mm]
&= F_{m-2j-k+2}F_{m-k}F_{-2j+1}F_{k+1}+F_{m-k+1}(F_{k+1}F_{m-2j-k+2}F_{-2j+2} + F_{m+1})\\[2mm]
&= F_{2m-2j-2k+3}F_{-2j+1}F_{k+1}-F_{m-2j-k+3}F_{m-k+1}F_{-2j+1}F_{k+1}+F_{m-k+1}(F_{k+1}F_{m-2j-k+2}F_{-2j+2} + F_{m+1})\\[2mm]
&= F_{2m-2j-2k+3}(F_{2j+k-1} + F_{-2j+2}F_{k})\\
& \qquad + F_{m-k+1}(-F_{m-2j-k+3}F_{-2j+1}F_{k+1}+F_{k+1}F_{m-2j-k+2}F_{-2j+2} + F_{m+1})\\[2mm]
&= F_{2m-2j-2k+3}F_{2j+k-1} + F_{2m-2j-2k+3}F_{-2j+2}F_{k}\\
& \qquad + F_{m-k+1}(-F_{m-2j-k+3}F_{-2j+1}F_{k+1}+F_{k+1}F_{m-2j-k+2}F_{-2j+2} + F_{m+1})\\[2mm]
\end{aligned}\]

We note that since 
\begin{multline}\label{mplus1}
F_{m+1} = F_{m-2j-k+3}F_{2j+k-1} + F_{m-2j-k+2}F_{2j+k-2}\\
=F_{m-2j-k+3}(F_{k+1}F_{2j-1} + F_k F_{2j-2}) + F_{m-2j-k+2}(F_{2j-2}F_{k+1} + F_{2j-3}F_k)
\end{multline}
 then 
\[\begin{aligned}
A 
&= F_{2m-2j-2k+3}F_{2j+k-1} + F_{2m-2j-2k+3}F_{-2j+2}F_{k}+F_{m-k+1}(F_{m-2j-k+3} F_k F_{2j-2} + F_{m-2j-k+2}F_{2j-3}F_k)\\[2mm]
&= F_{2m-2j-2k+3}F_{2j+k-1}+F_k \left[F_{-2j+2}( F_{2m-2j-2k+3}-F_{m-k+1}F_{m-2j-k+3} 
)+F_{m-k+1} F_{m-2j-k+2}F_{2j-3}\right]\\[2mm]
&= F_{2m-2j-2k+3}F_{2j+k-1}+F_k\left[ F_{-2j+2}F_{m-k}F_{m-2j-k+2}+F_{m-k+1} F_{m-2j-k+2}F_{2j-3}\right]\\[2mm]
&= F_{2m-2j-2k+3}F_{2j+k-1}+F_kF_{m-2j-k+2}( F_{-2j+2}F_{m-k}+F_{m-k+1} F_{3-2j})\\[2mm]
&= F_{2m-2j-2k+3}F_{2j+k-1}+F_kF_{m-2j-k+2}(F_{m-2j-k+3})\\[2mm]
\end{aligned}\]
\end{proof}

\begin{proposition}\label{app:simpB}
\begin{equation}\label{eq:AppSideB}
F_kF_{m-2j-k+3}^2 + F_{m+1}F_{m-k}  =F_{2j+k-2}  F_{2m-2j-2k+3}+ F_{k+1}F_{m-2j-k+3} F_{m-2j-k+2}
\end{equation}
\end{proposition}
\begin{proof}
We use the relations  $F_{m-2j-k+3} = F_{m-k+1}F_{-2j+3} + F_{m-k}F_{-2j+2}$, $F_{2m-2j-2k+3} = F_{m-2j-k+3}F_{m-k+1} + F_{m-k}F_{m-2j-k+2}$, and $F_{2j+k-2} = F_{-2j+3}F_{k} - F_{k+1}F_{-2j+2}$. Let $B$ equal the left-hand side of~\eqref{eq:AppSideB}.

\[\begin{aligned}
B &=  F_kF_{m-2j-k+3}(F_{m-k+1}F_{-2j+3} + F_{m-k}F_{-2j+2}) + F_{m+1}F_{m-k}\\[2mm]
&= F_kF_{m-2j-k+3}F_{m-k+1}F_{-2j+3} + F_{m-k}(F_kF_{m-2j-k+3}F_{-2j+2} + F_{m+1})\\[2mm]
&= F_kF_{-2j+3} (F_{2m-2j-2k+3}-F_{m-k}F_{m-2j-k+2})+ F_{m-k}(F_kF_{m-2j-k+3}F_{-2j+2} + F_{m+1})\\[2mm]
&= (F_{2j+k-2} +F_{k+1}F_{-2j+2} )F_{2m-2j-2k+3}+ F_{m-k}(F_kF_{m-2j-k+3}F_{-2j+2} + F_{m+1}-F_kF_{-2j+3}F_{m-2j-k+2})\\[2mm]
&= F_{2j+k-2}F_{2m-2j-2k+3}+ F_{k+1}F_{-2j+2} F_{2m-2j-2k+3}\\[1mm]
& \qquad + F_{m-k}(F_kF_{m-2j-k+3}F_{-2j+2} + F_{m+1}-F_kF_{-2j+3}F_{m-2j-k+2})\\[2mm]
\end{aligned}\]

We use \eqref{mplus1} again, to verify that 

\[\begin{aligned}
B 
&= F_{2j+k-2}  F_{2m-2j-2k+3}+ F_{k+1}F_{-2j+2} F_{2m-2j-2k+3} \\[1mm]
&\qquad +F_{m-k}(F_{m-2j-k+3}F_{k+1}F_{2j-1} + F_{m-2j-k+2}F_{2j-2}F_{k+1})\\[2mm]
&= F_{2j+k-2}  F_{2m-2j-2k+3}+ F_{k+1}F_{-2j+2} ( F_{m-2j-k+3}F_{m-k+1} + F_{m-2j-k+2}F_{m-k})\\[1mm]
& \qquad +F_{m-k}(F_{m-2j-k+3}F_{k+1}F_{2j-1} + F_{m-2j-k+2}F_{2j-2}F_{k+1})\\[2mm]
&= F_{2j+k-2}  F_{2m-2j-2k+3}+ F_{m-2j-k+3} ( F_{k+1}F_{-2j+2}F_{m-k+1} +F_{m-k}F_{k+1}F_{2j-1})\\[1mm]
& \qquad +F_{m-2j-k+2}(F_{k+1}F_{-2j+2} F_{m-k} +F_{m-k} F_{2j-2}F_{k+1})\\[2mm]
&= F_{2j+k-2}  F_{2m-2j-2k+3}+ F_{m-2j-k+3}F_{k+1} (F_{-2j+2}F_{m-k+1} +F_{m-k}F_{1-2j})\\[2mm]
&= F_{2j+k-2}  F_{2m-2j-2k+3}+ F_{k+1}F_{m-2j-k+3} F_{m-2j-k+2}\\[2mm]
\end{aligned}\]

\end{proof}
\end{appendices}
\bibliography{references}{}
\bibliographystyle{plain} 

\end{document}